\documentclass[abstracton, paper=a4, fontsize=12pt,DIV=12,bibliography=totoc,reqno]{amsart}

\usepackage[utf8]{inputenc}
\usepackage{amsthm, amssymb, amsmath, amsfonts, mathrsfs, dsfont, esint,textcomp}
\usepackage[english]{babel}
\usepackage[backend=biber,giveninits,maxnames=4,style=alphabetic,isbn=false, date=year, doi=false, eprint= true, url= false,sorting=ynt,sortcites=true]{biblatex}
\usepackage{csquotes}
\usepackage[shortlabels]{enumitem}
\usepackage[T1]{fontenc}
\usepackage[left=2cm,right=2cm,top=3.5cm,bottom=3.5cm]{geometry}
\usepackage[colorlinks=true, pdfstartview=FitV, linkcolor=blue, citecolor=blue, urlcolor=blue,pagebackref=false]{hyperref}
\usepackage{lmodern}
\usepackage{tikz}
\usetikzlibrary{hobby}

\definecolor{labelkey}{gray}{.8}
\definecolor{refkey}{gray}{.8}
\definecolor{darkblue}{rgb}{0,0,0.7} 
\definecolor{darkred}{rgb}{0.9,0.1,0.1}
\definecolor{darkgreen}{rgb}{0,0.5,0}

\newtheorem{thm}{Theorem}[section]
\newtheorem{theorem}[thm]{Theorem}
\newtheorem{prop}[thm]{Proposition}
\newtheorem{lem}[thm]{Lemma}

\theoremstyle{remark}
\newtheorem{rem}[thm]{Remark}

\theoremstyle{definition}
\newtheorem{defi}[thm]{Definition}

\newcommand{\ls}{\lesssim}

\newcommand{\e}{\mathbf{e}}

\newcommand{\1}{\mathbf{1}}
\newcommand{\R}{\mathbb{R}}
\newcommand{\Z}{\mathbb{Z}}

\newcommand{\eps}{\varepsilon}

\newcommand{\dd}{\, \mathrm{d}}
\newcommand{\wto}{\rightharpoonup}

\newcommand{\vc}{\mathbf}
\newcommand{\vu}{\vc{u}}
\newcommand{\vv}{\vc{w}}
\newcommand{\Oe}{{\Omega_\eps}}
\newcommand{\obst}{\mathcal{O}}

\renewcommand{\leq}{\leqslant}
\renewcommand{\geq}{\geqslant}
\renewcommand{\S}{\mathbb{S}}
\renewcommand{\subset}{\subseteq}

\DeclareMathOperator{\curl}{curl}
\DeclareMathOperator{\Id}{Id}
\DeclareMathOperator{\Tr}{Tr}
\DeclareMathOperator{\dv}{div}
\DeclareMathOperator{\dist}{dist}

\DeclareMathOperator*{\supp}{supp}

\numberwithin{equation}{section}
\allowdisplaybreaks

\bibliography{bib.bib}

\usepackage[LGR, T1]{fontenc} 
\newcommand{\koppa}{%
  \mathord{\small \text{\usefont{LGR}{STEP-TLF}{m}{n}\symbol{"15}}}%
}

\usetikzlibrary{arrows.meta}

\title{Quantitative homogenization of the compressible Navier-Stokes equations towards Darcy's law}
\date{\today}

\author{Richard M. H\"ofer}
\address{Faculty of Mathematics, University Regensburg, Universit\"atsstraße 31, 93053 Regensburg, Germany.}
\email{richard.hoefer@mathematik.uni-regensburg.de}

\author{\v{S}\'arka Ne\v{c}asov\'a}
\address{Institute of Mathematics, Czech Academy of Sciences,
\v Zitn\'a 609/25, 115 67 Praha 1, Czech Republic.}
\email{matus@math.cas.cz}

\author{Florian Oschmann}
\address{Institute of Mathematics, Czech Academy of Sciences,
\v Zitn\'a 609/25, 115 67 Praha 1, Czech Republic.}
\email{oschmann@math.cas.cz}

\begin{document}

\begin{abstract}
We consider the solutions $\rho_\eps, \vu_\eps$ to the compressible Navier-Stokes equations (NSE) in a domain  periodically perforated by holes of diameter $\eps>0$. We focus on the case where the diameter of the holes is of the same order as the distance between neighboring holes. This is the same setting investigated in the paper by Masmoudi [\url{http://www.numdam.org/article/COCV_2002__8__885_0.pdf}], where convergence $\rho_\eps, \vu_\eps$ of the system to the porous medium equation has been shown. 
We prove a quantitative version of this convergence result provided that the solution of the limiting system is sufficiently regular. The proof builds on the relative energy inequality satisfied by the compressible NSE.
\end{abstract}

\maketitle

\tableofcontents

 \section{Introduction}

The effect of small particles (also called obstacles or holes) in an incompressible or compressible fluid has attracted considerable interest during the last decades. The question of whether and, if yes, in which way many particles affect the fluid goes back to Darcy's  experimental studies \cite{Darcy1856} and has been addressed in a rigorous mathematical way via homogenization since the work of Tartar \cite{Tartar80}. The main outcomes are heuristically the following: 
 \begin{enumerate}
     \item Tiny holes are not felt by the fluid, hence the equations stay the same.
     \item For critically sized holes,  an extra friction term  of the same order of magnitude as internal friction arises as first proposed by Brinkman \cite{Brinkman1949}.
    \item For larger particles, the additional friction term dominates over internal friction, thus leading to Darcy's law, which can be written in terms of a porous medium equation for compressible flows. 
 \end{enumerate} 
The rigorous proofs of these results for stationary \emph{incompressible} Stokes and Navier-Stokes equations (NSE) were given by Tartar \cite{Tartar80}, Allaire  \cite{Allaire89, Allaire90a, Allaire90b}, and later for the evolutionary case by  Mikeli\'c \cite{Mikelic91}, Feireisl, Namlyeyeva, and Ne\v{c}asov\'{a} \cite{FeireislNamlyeyevaNecasova2016}, Lu and Yang \cite{LuYang2023} (see also \cite{BerezhnyiBerlyandKhruslov2008} and \cite{Khruslov2021}). Without being exhaustive, we also refer to related results for more general (random) particle configurations and inhomogeneous Dirichlet boundary conditions \cite{BeliaevKozlov96, CarrapatosoHillairet20, DesvillettesGolseRicci08, Giunti21, GiuntiHoefer19, HillairetMoussaSueur19, HoeferJansen20}, very large Reynolds numbers \cite{Hoefer2022}, and non-Newtonian fluids \cite{LuQian2023}.

Similar homogenization results for \emph{compressible} fluids are rather rare and mainly focus on the case of tiny particles which do not affect the limiting equations. Stationary NSE in three spatial dimensions have been considered by Feireisl and Lu in \cite{FeireislLu2015}. Building on this, results for a wider class of pressure functions \cite{DieningFeireislLu17}, for two-dimensional flows \cite{NecasovaPan2022}, for instationary flows \cite{LuSchwarzacher18, NecasovaOschmann2023, OschmannPokorny2023}, for randomly distributed holes \cite{BellaOschmann2023}, and for heat conducting fluids \cite{BasaricChaudhuri2023, LuPokorny2021, Oschmann22, PokornySkrisovsky2021} have been achieved, just to name a few. 

The first rigorous result on homogenization of compressible NSE, however, has been given by Masmoudi in \cite{Masmoudi02}, where he considered the case of large particles, the radius of which is proportional to their mutual distance. He gave a qualitative convergence result to the porous medium equation. To the authors' knowledge, this is the only available homogenization result for the compressible NSE where the particles are large enough to affect the limit which is still compressible (see however \cite{BellaOschmann2022, HoeferKowalczykSchwarzacher21} for homogenization results with a combined low Mach number limit).  In the case of the Navier-Stokes-Fourier equations, a similar result has been obtained by Feireisl, Novotn\' y,  and Takahashi \cite{FeireislNovotnyTakahashi2010}. An analogous problem for a quasi-static nonlocal version of the Navier-Stokes-Korteweg system has been considered in \cite{RohdeVonWolff20} where a nonlocal Cahn-Hilliard system is obtained in the homogenization limit.

The purpose of the present contribution is to revisit the setting in \cite{Masmoudi02} and to give a quantitative version of the convergence result to the porous medium equation. Convergence rates are an important aspect of homogenization theory in order to provide links to applications. In the case of the incompressible Stokes equations with  particle distances of the same order as their radii, such quantitative homogenization results have been achieved in \cite{Marusic-PalokaMikelic1996, Shen2022}.
We combine classical homogenization correctors  with the method of relative energies to obtain similar results for the compressible NSE. In particular, our methods differ substantially from the ones in \cite{Masmoudi02} which are mostly based on compactness and duality arguments. We provide results both on the three dimensional torus and bounded domains. The convergence rates we obtain seem suboptimal which is caused by the lack of regularity of the pressure. Our results hold for adiabatic exponents $\gamma \geq 2$ and a large range of  Reynolds numbers.

\subsection{Setting}
Let $\Omega$ be either a smooth bounded domain in $\R^3$, or $\Omega = \mathbb{T}^3 = \R^3/\Z^3$ be the three-dimensional (flat) torus. Denote $Q = (-1,1)^3$ and $B_1 \subset Q$ the unit ball. Let $\obst \Subset B_1$, the reference particle, be a fixed closed simply connected set with $0\in \obst$ as an inner point and smooth boundary such that $B_1 \setminus \obst$ is a connected open set.

\begin{figure}[h] 
\begin{center}

\begin{tikzpicture} [scale = 0.4]
\filldraw[fill={rgb:black,1;white,2}] (0,0) to [closed, curve through = {(1.2,2) (0.4,3.3) (1.9,4.8)  (4.9,4.2)(3.2,2.4) (4,0) (3,-1.4) (0.3,-1.2)}] (0,0);
\node (A) at (2.2:1.5) [above right] {$\obst$};
 
\node (B) at (-2.5, 4.0) [above right] {$Q$};
 
\draw[draw=black] (-2.8,-3.5) rectangle ++(10,10);
 
\end{tikzpicture}
\end{center}
\caption{A single cell}  \label{fig:single.particle}
\end{figure}

For $\eps>0$, we cover the set $\Omega$ with a regular mesh of size $2 \varepsilon$ and we denote by $x_i^{\varepsilon}$ the center of the cell with index $i$ at level $\varepsilon$. Let $Q_i^{\varepsilon} = x_i^{\varepsilon} + \eps Q$ be the cell with center $x_i^{\varepsilon}$, and $ i \in \{ 1, \dots, N(\varepsilon)\}$ be those indices for which the cell $Q_i^{\varepsilon}$ is entirely included in $\Omega$. Consider in each of the cells the particle
\begin{equation*}
\obst_i^{\varepsilon} := x_i^{\varepsilon} + \varepsilon \obst , \hspace{2 ex} i = 1 \dots N(\varepsilon).
\end{equation*}
Now we define the perforated domain $\Oe$ by
\begin{equation}\label{omegaeps}
\Omega_\eps = \Omega \setminus \bigcup_{ i = 1}^{N(\varepsilon)} \obst_i^{\varepsilon}.
\end{equation}
All the objects defined above are illustrated in Figure \ref{fig:perforated.domain}.
We remark that such a perforated domain can be viewed as a toy model for a porous medium. In more realistic models, both the fluid and the solid phase (the holes) are connected (see \cite{Allaire89} for such homogenization result for Stokes flows).

 \begin{figure}[h]
\includegraphics{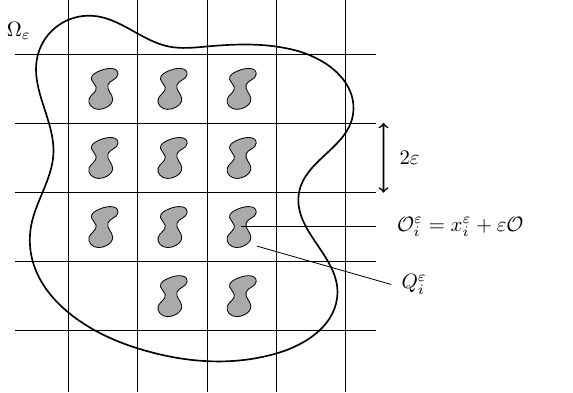}


 
  
 
  
  

\caption{The perforated domain $\Omega_\eps$} \label{fig:perforated.domain}
\end{figure}
If $\Omega = \mathbb T^3$, we will conveniently restrict ourselves to consider only $\eps>0$ with $(2\eps)^{-1} \in  \mathbb{Z}$. In this way, we have
\begin{align*}
    \Omega = \bigcup_{ i = 1}^{N(\varepsilon)} Q_i^\eps.
\end{align*}
Let $\mu >0$ and $\eta \geq 0$ be fixed constants and denote for any matrix $A$
\begin{align*}
        \S (A) =  \mu\left(A + A^T - \frac 2 3 \Tr(A) \Id \right) + \eta \Tr(A) \Id.
\end{align*}
For simplicity, we set $\mu=1$. For $T>0$, we consider the  compressible NSE
\begin{align}\label{NSE}
    \left\{
    \begin{array}{rcll}
         \partial_t \rho_\eps + \dv(\rho_\eps \vu_\eps) &=&0 & \qquad  \text{in } (0,T) \times \Omega_\eps, \\
          \eps^\lambda \big( \partial_t (\rho_\eps \vu_\eps) +  \dv(\rho_\eps \vu_\eps \otimes \vu_\eps) \big) - \eps^2 \dv \S(\nabla \vu_\eps) + \nabla p(\rho_\eps) &=& \rho_\eps \vc{f} & \qquad  \text{in } (0,T) \times \Omega_\eps , \\
          \vu_\eps &=& 0 & \qquad  \text{on } (0,T) \times \partial \Omega_\eps, \\
          \rho_\eps(0,\cdot) = \rho_{\eps0},   \quad (\rho_\eps\vu_\eps)(0,\cdot) &=& \vc{m}_{\eps0} & \qquad  \text{in } \Omega_\eps, \\
    \end{array}
    \right.
\end{align}
 for some given initial data $\rho_{\eps0}, \vc{m}_{\eps0}$ and $\lambda>0$. We assume the pressure as a function $p \in C([0,\infty)) \cap C^2((0,\infty))$ satisfies for some $\gamma>\frac32$
\begin{align}\label{p}
    p(0)=0, \qquad p'(s)>0 ~~ (\text{for } s>0), \qquad \lim_{s \to \infty} \frac{p'(s)}{s^{\gamma-1}} = p_\infty > 0.
\end{align}
The target system is Darcy's law for the fluid, and the continuity equation for the density, i.e.,
\begin{align}\label{limit.system}
    \left\{
    \begin{array}{rcll}
    \theta \partial_t \rho + \dv(\rho \vu) &=& 0 & \qquad  \text{in } (0,T) \times \Omega,\\
    \mathcal K^{-1} \vu + \nabla p(\rho) &=& \rho \vc{f} & \qquad  \text{in } (0,T) \times \Omega, \\
    (\rho \vu) \cdot \vc n &=& 0 &\qquad \text{on } (0,T) \times \partial \Omega, \\ 
    \rho(0,\cdot) = \rho_0,   \quad (\rho \vu)(0,\cdot) &=& \vc{m}_0 & \qquad  \text{in } \Omega. \\
        \end{array}
    \right.
\end{align}
Note that the boundary condition $(\rho \vu)\cdot \vc n = 0$ on $\partial \Omega$ is empty if  $\Omega = \mathbb T^3$, and we shall replace $\mathcal{K}^{-1}$ by $\mu \mathcal{K}^{-1}$ if the viscosity $\mu \neq 1$.
Here,
$\theta \in (0,1)$ is the porosity, i.e., the (asymptotic) volume fraction of the fluid domain,
\begin{align} \label{theta}
    \theta := \frac{|Q_i^\eps \setminus \obst_i^\eps|}{|Q_i^\eps|} = \frac{|Q \setminus \obst|}{|Q|} = 1 - \frac{|\obst|}{|Q|} .
\end{align}
Moreover, the resistance $\mathcal K \in  \R^{3\times 3}$ (resp.~the permeability $\mathcal K^{-1}$) is a constant positive definite matrix which is defined in \eqref{def:K} below.
Darcy's law reflects the high friction forces at the holes that lead to the effective term $\mathcal K^{-1} \vu$ which is dominant over both the fluid inertia and internal friction in the fluid.

The fluid velocity can be eliminated from this system, yielding the porous medium type equation
\begin{align}\label{FP}
\left\{
\begin{array}{rcll}
    \theta \partial_t \rho - \dv(\rho \mathcal K  \nabla p(\rho)) + \dv(\rho^2 \mathcal K  \vc{f}) &=& 0 & \qquad  \text{in } (0,T) \times \Omega,\\
    \big( \rho \mathcal{K}(\rho \vc f - \nabla p(\rho)) \big) \cdot \vc n &=& 0 & \qquad \text{on } (0,T) \times \partial \Omega,\\
    \rho(0, \cdot) &=& \rho_0 & \qquad \text{in } \Omega.
\end{array}
\right.
\end{align}
The formulation of the porous medium equation through this combination of Darcy's law and the continuity equation dates back to a work of Muskat \cite{Muskat34}.

At this point, we state the following informal version of our main result. For the precise statement, we refer to Theorem \ref{thm1} below.
\begin{thm}[Informal statement of the main result]
    Let $\gamma \geq 2$ and $\lambda > \lambda_0(\gamma) > 1$. Then, the solution $(\rho_\eps,\vu_\eps)$ of the primitive system \eqref{NSE} converges (in a suitable norm) to the solution $(\rho,\vu)$ of the target system \eqref{limit.system} with rate $\eps^\beta$, where $\beta > 0$ just depends on $\lambda$, $\Omega$, $\gamma$, and the initial data.
\end{thm}

\subsection{Elements of the proof} \label{sec:formal}
Our proof is based on the relative energy inequality for the primitive system, see Section \ref{sec:relativeEnergy}. We consider the relative energy between the solution of the primitive system $(\rho_\eps,\vu_\eps)$ and a modification of the solution $(\rho,\vu)$ of the target system. The modification is needed to ensure no-slip boundary conditions for the modified fluid velocity and that the modification almost solves the momentum equation of the primitive system.
The modification is based on correctors, the choice of which is motivated by a formal two-scale expansion.
For convenience, we  review the formal argument that goes back to \cite{Diaz99}.

\medskip

We make the ansatz:
\begin{align*}
    \rho_\eps(x) &= \rho_0(x,x / \eps) + \eps \rho_1(x,x / \eps) + \dots,\\
    \vu_\eps(x) &= \vu_0(x,x / \eps) + \eps \vu_1(x,x / \eps) + \dots,
\end{align*}
where $\rho_k, \vu_k$ are periodic in the second variable, which is denoted by $y = x/\eps$. 
Then, we also have
\begin{align*}
    p_\eps(x) = p(\rho_\eps) = p_0(x,x/\eps) + \eps p_1(x,x/\eps) + \dots.
\end{align*}
We now insert this expansion into the equations 
and compare terms of the same order in $\eps$.
To order $\eps^{-1}$, we find (for $\lambda > 0$) 
\begin{align*}
    \dv_y (\rho_0 \vu_0) &= 0, \\
    \nabla_y p_0 &= 0.
\end{align*}
The second equation yields that $p_0$ and thus $\rho_0$ are $y$-independent and therefore (assuming $\rho_0 >0$) the first equation becomes
\begin{align}\label{div.2scale}
    \dv_y \vu_0 = 0.
\end{align}
Next, we have at order $\eps^0$ (for $\lambda > 1$ and using $\dv_y \vu_0 = 0$)
\begin{align}
    \partial_t \rho_0 + \dv_x (\rho_0 \vu_0) + \dv_y(\rho_0 \vu_1 + \rho_1 \vu_0) &= 0, \label{continuity.2scale}\\
    -\Delta_y \vu_0 + \nabla_x p_0 +  \nabla_y p_1 &= \rho_0 \vc f. \label{Stokes.2scale}
\end{align}
Since the no-slip boundary condition at the holes must be encoded in the fast variable, this is complemented by
\begin{align} \label{boundary.2scale}
    \vu_0(x,y) = 0 \quad \text{for } y \in \obst.
\end{align}
For fixed $x$, equations \eqref{div.2scale}, \eqref{Stokes.2scale}, and \eqref{boundary.2scale} are just a Stokes problem in $y$ in the fixed domain $Q \setminus \obst$ with periodic boundary conditions.
We thus introduce  $(\vv_i, q_i) \in H^1(Q) \times L^2(Q \setminus \obst )/\R$ to be the unique
solution to the Stokes problem\footnote{More precisely, the cube $Q$ should herein be replaced by the flat torus $\R^3 / (2\Z)^3$.}
\begin{align} \label{sol.Stokes}
\left\{
\begin{aligned}
-\Delta \vv_i + \nabla q_i = \vc e_i  &\quad \text{in } Q\setminus \obst, \\
 \dv \vv_i = 0 &\quad \text{in } Q \setminus \obst, \\
 \vv_i = 0 &\quad \text{in }  \obst, \\
 (\vv_i, q_i) &\quad \text{are } Q\text{-periodic}.
\end{aligned}
\right.
\end{align}
We set $W = (\vv_1, \vv_2, \vv_3)$, $\vc q = (q_1,q_2,q_3)^T$.
Then, the solution to \eqref{div.2scale}, \eqref{Stokes.2scale}, and \eqref{boundary.2scale} is given by
\begin{align}
    \vu_0(x,y) = W(y)(\rho_0(x) \vc f(x) - \nabla_x p_0(x)), \label{u_0} \\
    p_1(x,y) = \vc q(y) \cdot (\rho_0(x) \vc f(x) - \nabla_x p_0(x)). \label{p_1} 
\end{align}
The function $\vu_0$ is naturally extended by $0$ inside $\obst$. By periodicity in $y$ we then easily pass to the weak limit by just taking the average over $Q$, which yields
\begin{align} \label{Darcy.formal}
    \vu_0(x, x/\eps) \wto \vu(x) = \mathcal K(\rho_0 \vc f(x) - \nabla_x p_0(x)),
\end{align}
where 
\begin{align} \label{def:K}
    \mathcal K = \fint_Q W \dd x.
\end{align}

By classical arguments, $\mathcal{K}$ is positive definite (see \cite[Chapter~7, Proposition~2.2]{Sanchez-Palencia80} and \cite[end of Section~1]{Masmoudi02}). Moreover, $\mathcal{K}$ is symmetric by
\begin{align*}
         \mathcal{K}_{ij} = \fint_Q \vc e_i \cdot  W \vc e_j  \dd x =  \frac 1 {|Q|} \int_{Q \setminus \obst}    (-\Delta \vv_i + \nabla p_i) \cdot \vv_j \dd x =  \frac 1 {|Q|}\int_{Q \setminus \obst} \nabla \vv_i : \nabla \vv_j \dd x =  \mathcal{K}_{ji}.
\end{align*}

To pass to the limit in the continuity equation \eqref{continuity.2scale}, let $\tilde \rho_0$ denote the extension by $0$ in the holes, and observe that then the following equation is valid in all of $\Omega$:
\begin{align*}
    \partial_t \tilde \rho_0 + \dv_x (\rho_0 \vu_0) + \dv_y(\rho_0 \vu_1 + \rho_1 \vu_0) = 0.
\end{align*}
Note that there is no need to change $\rho_k$ in the holes for the terms inside of the divergence terms because all the fluid velocities $\vu_i$ vanish there.
Now we can again pass to the limit by just taking the average with respect to $y$ over the whole cell $Q$. Denoting  $\rho(x) = \rho_0(x,x/\eps)$, since $\rho_0$ is constant in the second variable, we  obtain 
\begin{align} \label{continuity.formal}
        \theta \partial_t  \rho + \dv_x (\rho \vu) = 0,
\end{align}
where we recall from \eqref{theta} that $\theta = 1 - |\obst|/|Q|$ is the fluid volume fraction.
Equations \eqref{Darcy.formal} and \eqref{continuity.formal} are the target equations.

\medskip 

If $\Omega = \mathbb R^3$, this motivates to take in the relative energy  inequality the functions $(r_\eps, \vv_\eps)$ defined as $\vv_\eps(x)  = \vu_0(x,x/\eps)$ and  $r_\eps(x) =  p^{-1}(p_0(x,x/\eps) + \eps p_1(x,x/\eps))$. By \eqref{u_0}, \eqref{p_1}, and \eqref{Darcy.formal}, we have
\begin{align} \label{r.w.formal}
\begin{split}
\vv_\eps(x) & = \vu_0(x,x/\eps) = W(x/\eps)(\rho_0(x) \vc f(x) - \nabla_x p_0(x)) = W(x/\eps) \mathcal K^{-1} \vu(x),  \\
    r_\eps(x) &=  p^{-1}(p_0(x,x/\eps) + \eps p_1(x,x/\eps)) = p^{-1}\big( p(\rho(x)) + \eps  \vc q(x/\eps) \cdot \mathcal K^{-1} \vu(x) \big) .
\end{split}
\end{align}
It might seem surprising at first that we take different orders of approximation for the density and the fluid velocity, namely first order in $\eps$ for the density and zeroth order for the fluid velocity.
This choice is motivated, however, by the terms  appearing in \eqref{Stokes.2scale}. In particular, the first order pressure correction corresponds to the pressure appearing in the cell problem \eqref{sol.Stokes} which is at the heart of the additional friction term in the limit system.

\subsection{Comparison with previous results and open problems}\label{sec:previous}

The convergence $(\rho_\eps, \vu_\eps) \to (\rho, \vu)$ has been shown qualitatively in \cite{Masmoudi02} for $\lambda = 4$. 
At first glance, the scaling in \cite{Masmoudi02} appears to be different.
However, if we denote the corresponding fluid velocity in \cite{Masmoudi02} by $\tilde{\vu}_\eps$,  then,  convergence of $\vu_\eps := \eps^{-2} \tilde{\vu}_\eps$ is shown in \cite{Masmoudi02}, and $\vu_\eps$ satisfies precisely \eqref{NSE} with $\lambda=4$. 
We choose here to directly rewrite the system for $\tilde{\vu}_\eps$ anticipating how the holes slow down the fluid velocity $\tilde{\vu}_\eps$. 
We believe that this choice makes more transparent the scaling of the dimensionless quantities.
In the present paper, we provide  quantitative convergence results with  the help of the relative energy method.

We remark that the scaling in the NSE \eqref{NSE} corresponds to a Reynolds number $\mathrm{Re} = \eps^{\lambda-2}$, a Mach number $\mathrm{Ma} = \eps^{\lambda/2}$, a Froude number $ \mathrm{Fr} = \eps^{\lambda/2}$, and a Strouhal number $\mathrm{Sr} = 1$. 
In particular, the Knudsen number satisfies $\mathrm{Kn} \sim \mathrm{Ma}/\mathrm{Re} = \eps^{2 - \lambda/2}$. The Knudsen number is the ratio between the mean free path length $l$ and the observation length scale $L$. The latter is chosen to be the macroscopic length scale, hence of order $1$. In order that the fluid is reasonably modeled by the Navier-Stokes equations at the length scale of the holes, we must have $\mathrm{Kn} = l/L \lesssim \eps$ as $\eps \to 0$ and thus $\lambda \leq 2$. As stated in the main theorem, Theorem \ref{thm1}, we are  able to show the convergence result for $\lambda \leq 2$ if the adiabatic exponent $\gamma \geq 3$. To our knowledge, this is the first instance of a homogenization result for compressible fluids where the mean free path length is actually less than the diameter of the holes. As remarked above there is natural threshold $\lambda > 1$ since for $\lambda \leq 1$, the Reynolds number at the particle length scale $\eps$ is at least of order one and therefore the validity of the local Stokes problem \eqref{Stokes.2scale} ceases to be valid even formally.
We are able to asymptotically reach that threshold as $\gamma \to \infty$, see \eqref{lambda_0}.

\smallskip

 Compared to \cite{Masmoudi02}, where  $\gamma \geq 3$ is assumed,  our result is valid for all adiabatic exponents $\gamma \geq 2$ on the expense of a larger $\lambda$. In view of applications, it would be very important to further relax this assumption. We recall that existence of renormalized finite energy weak solutions (see Definition \ref{defin:WeakSolTime}) is known for $\gamma > 3/2$. Moreover, in \cite{FeireislJinNovotny2012}, the relative energy has been used to show weak-strong uniqueness for all $\gamma > 3/2$. Nevertheless, our proof cannot be easily extended to $ \gamma \in (3/2,2)$ (and this shortcoming on the range of adiabatic exponents seems presently common to all related homogenization results for the compressible Navier-Stokes equations). The reason for that is the very weak dissipation in the relative energy inequality \eqref{relen}, which has a prefactor $\eps^2$. Thanks to the Poincar\'e inequality \eqref{Poinc}, this still allows to absorb $\|\vu_\eps - \vv_\eps\|_{L^p(\Oe)}^2$ for $p=2$. However, that is no longer possible for $p>2$ which seems necessary when $\gamma < 2$ (e.g., to estimate the term $I_2$ in Section~\ref{sec:estimate.errors}).

 \smallskip

We emphasize that even though $\mathrm{Ma} = \eps^{\lambda/2}$ vanishes as $\eps$ does, the limit system is still compressible. The reason for this behavior is that both the external forces and friction forces at the holes, which are manifested in the term $\mathcal K^{-1} \vu$ in \eqref{limit.system}, are sufficiently large to allow for a pressure induced by density fluctuations of order $1$. 

\smallskip


\begin{figure}[h]
    \begin{tikzpicture}
        \draw[fill=orange!15, draw=none] (-5,1) rectangle (4.5,4.5);
        \path[fill=yellow!30, draw=none] (-5,1) -- (1,1) -- (2,2) -- (-5,2) -- (-5,1);
        \path[fill=cyan!30, draw=none] (0,1) -- (1,1) -- (2,2) -- (0,1);
        
        \draw[dashed] (0,0) -- (1,1);
        \draw[dashed] (1,0) -- (1,1);
        \draw[dashed] (2,0) -- (2,2);
        
        \draw[very thick, orange] (1,1) -- (4.5,4.5);
        \draw[very thick, orange] (1,1) -- (4.5,1);
        \draw[very thick, orange] (-5,2) -- (4.5,2);
        \draw[very thick, brown] (-5,1) -- (1,1);

        \draw[-Stealth] (-5, 0) -- (5,0);
        \node at (5,0) [anchor=west] {$\koppa$};
        \draw[-Stealth, gray] (0,0) -- (0,5);
        \node at (0,5) [anchor=east] {$\alpha$};
        \node at (0,0) [anchor=north] {$0$};
        \node at (1,0) [anchor=north] {$1$};
        \node at (2,0) [anchor=north] {$3/2$};

        \node at (-5,1) [anchor=east] {$1$};
        \node at (-5,2) [anchor=east] {$3/2$};

        \node[fill=orange!15, text=black] at (0.2, 3.3) {\small Not locally incompressible};
        \node[align=left] at (2.3, 2) [anchor=west] {\small Big local\\ Reynolds number};
        \node[fill=white, text=brown] at (-1.5, .95) [anchor=north] {\small Darcy's law (periodic)};
        \node at (-5, 2) [anchor=north west] {\small Big local Knudsen number};
        \node at (-5, 1) [anchor=south west] {\small Darcy's law};

        \node[fill=orange!15, text=cyan] at (0, 2.05) [anchor=south] {\small Physically relevant};
        \draw[-Stealth, cyan] (-.1,2.15) -- (0.75,1.2);
    \end{tikzpicture}
    \caption{Formal scaling regimes} \label{fig:regimes}
\end{figure}

Despite considerable effort in the study of such homogenization problems for the compressible NSE, the case of holes of size $\eps^{\alpha}$ for $\alpha \in (1,3]$ is currently completely open. More precisely, when $\alpha \in [1,3]$, one can study the primitive system
\begin{align}\label{NSE.general}
    \left\{
    \begin{array}{rcll}
         \partial_t \rho_\eps + \dv(\rho_\eps \vu_\eps) &=&0 & \qquad  \text{in } (0,T) \times \Omega_\eps, \\
          \big( \partial_t (\rho_\eps \vu_\eps) +  \dv(\rho_\eps \vu_\eps \otimes \vu_\eps) \big) - \eps^{\koppa}  \dv \S(\nabla \vu_\eps) +  \eps^{\kappa_1} \nabla p(\rho_\eps) &=& \eps^{\kappa_2}  \rho_\eps \vc{f} & \qquad  \text{in } (0,T) \times \Omega_\eps , \\
          \vu_\eps &=& 0 & \qquad  \text{on } (0,T) \times \partial \Omega_\eps, \\
          \rho_\eps(0,\cdot) = \rho_{\eps0},   \quad (\rho_\eps\vu_\eps)(0,\cdot) &=& \vc{m}_{\eps0} & \qquad  \text{in } \Omega_\eps.
    \end{array}
    \right.
\end{align}
Formally, as long asymptotically a linear drag relation prevails at the holes, which is proportional to the hole size $\eps^\alpha$ and to $\eps^{\koppa}$, one expects the total drag to be of order $\eps^{{\koppa} + \alpha - 3}$ since $\eps^{-3}$ is the number density of holes. To match the largest term, a convenient choice is $\kappa_1 = \kappa_2 = \min\{0,{\koppa}, {\koppa} + \alpha - 3\}$. (For smaller values of $\kappa_1$, the limit equation as $\eps \to 0$ is expected to become incompressible.) By rescaling the equations to the $\eps^\alpha$ length scale of the holes, one obtains the condition ${\koppa} < \alpha$ for a small local Reynolds number, and the condition $\alpha < 3/2$ for local incompressibility. In particular, for ${\koppa} < \alpha$ and $\alpha < 3/2$, one expects Darcy's law as homogenization limit as $\eps \to 0$ (see also \cite[Conjecture~1.1]{HoeferKowalczykSchwarzacher21}).\footnote{With a different permeability $\mathcal K^{-1}$ that corresponds to the mobility of a single hole in the whole space $\R^3$.}
Outside of this regime, it seems very hard to obtain a limit system, even formally.
The physical condition (which is not reflected in any mathematical difficulty for the homogenization problem) of the mean free path length to be small compared to the small length scale $\eps^\alpha$ gives rise to the additional constraint $\koppa - 3(\alpha - 1)  \geq 0$.  The different regimes are depicted in Figure~\ref{fig:regimes}. Our present result corresponds to $\alpha =1$ and any values ${\koppa} = 2 - \lambda  < 1$ (depending on $\gamma$).

These scaling regimes might be compared to the ones for the incompresssible Navier-Stokes equations (see \cite[Figure~1]{Hoefer2022}), which are much richer because the condition $\alpha < 3/2$ only arises due to compressibility. In particular, it does not seem to be possible to obtain a Brinkman type compressible homogenization limit.
Unfortunately, the relative energy method that we use here does not directly apply to cases $\alpha > 1$. However, we hope that our ansatz can eventually contribute to resolving this regime.  Here, we just point out why the current method breakes down for $\alpha \in (1,3/2)$: In this case, the functions $W$ and $\vc q$ should be replaced by those oscillating functions introduced by Allaire \cite{Allaire90a, Allaire90b}
to account for the large distance between the particles compared to their diameters. In particular,  $\|\vc q(\cdot / \eps^\alpha)\|_{L^\infty(\Omega)} \sim \eps^{-\alpha}$ and this implies that the $L^p$-norm of $\nabla r_\eps$ defined as in \eqref{r.w.formal} explodes as $\eps \to 0$ for all $p \in [1,\infty]$. However, boundedness of $\|\nabla r_\eps\|_{L^\infty(\Omega)}$ is crucially used in our proof, see the estimates of $I_4$ and $I_5$ in Section \ref{sec:estimate.errors}.

\subsection{Structure of the remainder of the paper}\label{sec:structure}
The rest of the paper is organized as follows. In Section \ref{sec:wksol}, we recall the concept of weak solutions to the primitive system \eqref{NSE}, introduce the relative energy, and state our main result. In Section \ref{sec:bds}, we obtain uniform bounds for the involved functions, as well as a pressure decomposition crucial for our analysis. Section \ref{sec:proof} is devoted to the proof of the main result in the case $\Omega = \mathbb T^3$. After some estimates on the corrector functions $\vv_\eps, r_\eps$ introduced above,  we use them as test functions in the relative energy and manipulate and estimate the error terms. Finally, in Section \ref{sec:bdDom} we explain how to adapt the proof from the torus setting to bounded domains. The main point is that in general $\vv_\eps \neq 0$ on $\partial \Omega$ such that $\vv_\eps$ is no longer an admissible test function in the relative energy. 

\medskip

\paragraph{\textbf{Notation:}} Throughout the paper, we will use the standard notations for Lebesgue and Sobolev spaces, and denote them even for vector- or matrix-valued functions as in scalar case, that is, $L^p(\Omega)$ instead of $L^p(\Omega; \R^3)$. We write $a \ls b$ if there is a constant $C>0$ that is independent of $a,b,$ and $\eps$ such that $a \leq C b$. The constant $C$ might change its value whenever it occurs. For two matrices $A,B \in \R^{3 \times 3}$, the Frobenius inner product is denoted by $A:B = {\rm tr}(A^T B) = \sum_{i,j=1}^3 A_{ij} B_{ij}$. For a function $f$ defined on a domain $D\subset \R^3$, we denote the mean value by $\fint_D f \dd x = \frac{1}{|D|} \int_D f \dd x$. Lastly, we define the divergence of a matrix $A \in \R^{3\times 3}$ column-wise as $(\dv A)_j = \dv(A \e_j)$ for $j \in \{1,3\}$, where $\e_j$ is the $j$-th canonical basis vector. Also, we define the gradient of a vector $a \in \R^3$ as $(\nabla a)_{ij} = \partial_i a_j = \partial a_j / \partial x_i$. We emphasize that this does not follow the standard definition of these operators, however, it leans our notation in order not to have transpositions on matrices. Additionally, this definition makes the gradient of a scalar function a column- rather than a row-vector, so no implicit convention has to be made to add the velocity $\vu$ and the pressure gradient $\nabla p$ in the systems under consideration.

\section{Notion of weak solutions and statement of the main result}\label{sec:wksol}

\subsection{Weak solutions of the primitive system}
In this section, we recall the concept of renormalized finite energy weak solutions. Global existence of  renormalized finite energy weak solutions can be shown as in \cite{Lions1998} and \cite{FeireislNovotnyPetzeltova2001}, see also \cite{Feireisl2002}.

\begin{defi}\label{defin:WeakSolTime}
Let $T>0$ be fixed, $\gamma>3/2$, and assume for the initial data
\begin{align*}
\rho_\eps(0,\, \cdot)=\rho_{\eps0},\quad (\rho_\eps\vu_\eps)(0,\, \cdot)=\vc{m}_{\eps0},
\end{align*}
together with the compatibility conditions
\begin{align*}
\rho_{\eps0}\geq 0 \text{ a.e.~in } \Omega_\eps,\quad \rho_{\eps0}\in L^\gamma(\Omega_\eps),\quad \vc{m}_{\eps0}=0\text{ whenever } \rho_{\eps0}=0,\quad \vc{m}_{\eps0}\in L^\frac{2\gamma}{\gamma+1}(\Oe),\quad  \frac{|\vc{m}_{\eps0}|^2}{\rho_{\eps0}}\in L^1(\Omega_\eps).
\end{align*}
We call a couple $(\rho_\eps,\vu_\eps)$ a \emph{renormalized finite energy weak solution} to system \eqref{NSE} in the time-space cylinder $(0,T)\times \Omega_\eps$ if:
\begin{itemize}
\item It holds
\begin{gather*}
\rho_\eps\geq 0 \text{ a.e. in $(0,T)\times \Omega_\eps$,}\quad \rho_\eps\in C(0,T;\, L_\mathrm{weak}^{\gamma}(\Omega_\eps)),\\
\int_{\Omega_\eps}\rho_\eps(\tau, \cdot) \dd x=\int_{\Omega_\eps} \rho_{\eps0} \dd x \ \text{ for any $\tau\in (0,T)$},\\
\vu_\eps\in L^2(0,T;\,W^{1,2}_0(\Omega_\eps)),\quad \rho_\eps\vu_\eps\in C(0,T;\, L_\mathrm{weak}^\frac{2\gamma}{\gamma+1}(\Omega_\eps)).
\end{gather*}
Here,  $h \in C(0,T;\, L_\mathrm{weak}^p(\Omega_\eps))$ means that the map $t \mapsto h(t,\cdot)$ is continuous in the weak topology  of $L^p(\Oe)$.

\item We have for any $0\leq \tau\leq T$ and any $\psi\in C_c^\infty([0,T)\times \Omega_\eps)$
\begin{align}\label{contEq}
\int_0^\tau \int_{\Omega_\eps}\rho_\eps \partial_t\psi+\rho_\eps\vu_\eps\cdot\nabla\psi \dd x \dd t = \int_{\Omega_\eps}\rho_\eps(\tau, \cdot)\psi(\tau, \cdot)\dd x-\int_{\Omega_\eps} \rho_{\eps0} \psi(0,\cdot) \dd x;
\end{align}

\item We have for any $0\leq\tau\leq T$ and any $\phi\in C_c^\infty([0,T)\times \Omega_\eps;\R^3)$
\begin{align}\label{eq:Momentum}
\begin{split}
&\int_0^\tau \int_{\Omega_\eps} \eps^\lambda(\rho_\eps\vu_\eps \cdot \partial_t \phi+\rho_\eps \vu_\eps\otimes \vu_\eps:\nabla\phi) - \eps^2 \S(\nabla \vu_\eps):\nabla\phi + p(\rho_\eps)\dv \phi + \rho_\eps \vc{f} \cdot \phi \dd x \dd t\\
&\qquad =\int_{\Omega_\eps} \eps^\lambda (\rho_\eps\vu_\eps)(\tau, \cdot)\phi(\tau, \cdot) \dd x-\int_{\Omega_\eps} \eps^\lambda \vc{m}_{\eps0}\phi(0,\, \cdot)\dd x;
\end{split}
\end{align}

\item The energy inequality
\begin{align}\label{EnergIneTime}
\begin{split}
&\int_{\Omega_\eps} \frac12 \eps^\lambda \rho_\eps|\vu_\eps|^2(\tau,\, \cdot)+ H(\rho_\eps) \dd x+\int_0^\tau\int_{\Omega_\eps} \eps^2 \S(\nabla\vu_\eps):\nabla\vu_\eps \dd x \dd t\\
&\qquad \leq \int_{\Omega_\eps} \eps^\lambda \frac{|\vc{m}_{\eps0}|^2}{2\,\rho_{\eps0}}+ H(\rho_{\eps0}) \dd x + \int_0^\tau\int_{\Omega_\eps} \rho_\eps\vc{f} \cdot \vu_\eps \dd x \dd t
\end{split}
\end{align}
holds for almost every $0\leq \tau\leq T$, where the so-called pressure potential $H$ is defined via
\begin{align}\label{defPressPot}
s H'(s)-H(s)=p(s), \qquad H(1) = 0.
\end{align}

\item The zero extension $(\tilde{\rho}_\eps,\tilde{\vu}_\eps)$ satisfies in $\mathcal{D}^\prime((0,T)\times\R^d)$
\begin{align}\label{eq:renormTime}
\partial_t \tilde{\rho}_\eps+\dv(\tilde{\rho}_\eps\tilde{\vu}_\eps)=0,\quad \partial_t b(\tilde{\rho}_\eps)+\dv(b(\tilde{\rho}_\eps)\tilde{\vu}_\eps)+(\tilde{\rho}_\eps b^\prime(\tilde{\rho}_\eps)-b(\tilde{\rho}_\eps))\dv \tilde{\vu}_\eps=0
\end{align}
for any $b\in C([0,\infty)) \cap C^1((0,\infty))$ with $|b'(z)z| \lesssim z^\varpi + z^\frac{\gamma}{2}$ for $z>0$ and a certain $\varpi \in (0, \gamma/2)$.
\end{itemize}
\end{defi}

\subsection{Relative energy}
\label{sec:relativeEnergy}

The proof of our main result is based on the relative energy inequality for solutions to the compressible NSE, which is sometimes also called relative entropy inequality. It has been shown in \cite{FeireislJinNovotny2012} that any renormalized finite energy weak solution satisfies the following relative energy inequality (see also \cite{FeireislNovotnySun2011}).

For smooth functions $(r,\vv)\in C^\infty([0,T]\times \overline \Oe)$ with $r>0$ in $[0,T]\times \overline \Oe$ and $\vv = 0$ on $[0,T] \times \partial \Oe$, we define the relative energy as
\begin{align} \label{def.E}
    E_\eps(\rho_\eps, \vu_\eps | r, \vv)(\tau) = \int_\Oe \frac12 \eps^\lambda \rho_\eps |\vu_\eps-\vv|^2 + H(\rho_\eps)-H'(r)(\rho_\eps-r)-H(r) \dd x,
\end{align}
where $H$ is as in \eqref{defPressPot}. Then, for almost any $\tau\in [0,T]$, the relative energy inequality reads
\begin{align}\label{relen}
\begin{aligned}
     E_\eps(\rho_\eps, \vu_\eps | r, \vv)(\tau)  + \int_0^\tau \int_\Oe \eps^2 &\big( \S(\nabla \vu_\eps) - \S(\nabla \vv) \big) : (\nabla\vu_\eps-\nabla\vv \big) \dd x \dd t \\
     &\leq  E_\eps(\rho_\eps, \vu_\eps | r, \vv)(0)  + \int_0^\tau \mathcal{R}_\eps(\rho_\eps,\vu_\eps | r, \vv) \dd t,
\end{aligned}
\end{align}
where the remainder $\mathcal{R}_\eps$ is given by
\begin{align*}
    \mathcal{R}_\eps(\rho_\eps,\vu_\eps | r,\vv) &= \int_\Oe \eps^\lambda \rho_\eps (\partial_t \vv + (\vu_\eps \cdot \nabla) \vv) \cdot (\vv - \vu_\eps) \dd x + \int_\Oe \eps^2 \S(\nabla\vv) : \nabla (\vv - \vu_\eps) \dd x \\
    &\quad + \int_\Oe \rho_\eps \vc{f} \cdot (\vu_\eps - \vv) \dd x + \int_\Oe (r-\rho_\eps)\partial_t H'(r) + \nabla H'(r)\cdot (r\vv - \rho_\eps\vu_\eps) \dd x \\
    &\quad - \int_\Oe \dv\vv (p(\rho_\eps)-p(r)) \dd x.
\end{align*}
Let us remark that the scaling in $\eps$ can be obtained straightforwardly from the form of the relative  energy in \cite{FeireislNovotnySun2011, FeireislJinNovotny2012} by a rescaling of the time and space variables.
 Let us also remark that the relative energy turns into the standard energy inequality \eqref{EnergIneTime} by setting $\vv=0$ and $r=\fint_\Oe \rho_\eps \dd x$ and using mass conservation. 

\begin{rem} \label{rem:rel.energy}
The regularity of $\vv$ and $r$ in the relative energy inequality can be weakened, see \cite[Section~4]{FeireislNovotnySun2011} and \cite[Section~3.2.2]{FeireislJinNovotny2012}. In particular, we can apply it with $(r,\vv) \in W^{1,\infty}((0,T) \times \Oe)$ with $r \geq \underline{r} > 0$ in $[0,T]\times \overline \Oe$ and $\vv = 0$ on $[0,T] \times \partial \Oe$.
\end{rem}

\subsection{Main result}

\begin{theorem}\label{thm1}
  Let $\Omega = \mathbb T^3$ or a smooth bounded domain in $\R^3$.  Let $T > 0$, $\gamma \geq  2$, $\lambda > \lambda_0$, $\vc f \in L^\infty((0,T)  \times \Omega)$, and let  $\Omega_\eps$ be as in \eqref{omegaeps} and $(\rho_{\eps0},\vc{m}_{\eps0})$ be as in Definition \ref{defin:WeakSolTime} such that
   \begin{align} \label{est.initial}
     \eps^\lambda \|\rho_{\eps 0} |\vu_{\eps 0}|^2\|_{L^1(\Oe)} + \|\rho_{\eps 0}\|_{L^\gamma(\Oe)}  \lesssim 1,
\end{align}
    where $\vu_{\eps 0}$ is such that $\rho_{\eps 0} \vu_{\eps 0} = \vc m_{\eps 0}$. Let $(\rho_\eps, \vu_\eps)$ be a renormalized finite energy weak solution to the NSE \eqref{NSE}, and let $\rho, \vu\in  W^{1,\infty}((0,T) \times \Omega)$ with $\vu \in L^{\infty}(0,T;W^{2,2}(\Omega)) \cap W^{1,\infty}(0,T;W^{1,\infty}(\Omega))$ be a strong solution to Darcy's law \eqref{limit.system} with $\rho \geq \underline{\rho} $ for some constant $\underline{\rho} > 0$. Then, there exists $\eps_0 > 0$  such that for all $\eps < \eps_0$
\begin{align} \label{est.thm}
\begin{aligned}
    &\|\rho_\eps - \rho\|_{L^\infty(0,T;L^2(\Oe))}^2  + \|\vu_\eps - \vu\|_{L^2(0,T; W^{-1,2}(\Omega))}^2 \\
    &\leq C \bigg(  \int_\Oe H(\rho_{\eps 0}) - H'(\rho_0)(\rho_{\eps 0} - \rho_0) - H(\rho_0) \dd x +  \eps^\lambda \|\rho_{\eps 0} |\vu_{\eps 0}|^2\|_{L^1(\Oe)} + \eps^{ \beta} \bigg),
    \end{aligned}
\end{align}
where  $C>0$ is an absolute constant which is independent of $\eps$, and where $\beta>0$ and $\lambda_0>1$ are given by
    \begin{align} \label{lambda_0}
        \lambda_0 &= \begin{cases}
            1 + \frac3\gamma & \text{if } \gamma \geq 3,\\
            \frac53 + \frac1\gamma & \text{if } 2 \leq \gamma < 3,
        \end{cases}\\
        \beta &= \min \bigg\{1, 2\lambda-2, \lambda-\frac3\gamma \bigg\} \ \text{ if } \ \Omega = \mathbb T^3,\\
        \beta &= \min \bigg\{\frac12, 2\lambda-2, \lambda-\frac3\gamma \bigg\} \ \text{ if $\Omega \subset \R^3$ is a bounded domain}.
    \end{align}
    Moreover, if $2 \leq \gamma < 3$, we require  the initial data additionally to satisfy
    \begin{align} \label{zusatz}
        \eps^{2\lambda - 5} \bigg( \int_\Oe H(\rho_{\eps 0}) - H'(\rho_0)(\rho_{\eps 0} - \rho_0) - H(\rho_0) \dd x + \eps^\lambda \|\rho_{\eps 0} |\vu_{\eps 0}|^2\|_{L^1(\Oe)} + \eps^\beta \bigg) \leq \delta_0,
    \end{align}
   for some $\delta_0>0$ sufficiently small but independent of $\eps$.
\end{theorem}

\begin{rem}
    \begin{enumerate}[(i)]
    \item Note that the condition $\lambda > \lambda_0= \frac53 + \frac1\gamma$ for $\gamma < 3$ is actually redundant since \eqref{zusatz} enforces that. We remark that our proof shows that for $2 \leq \gamma \leq 3$ the assertion remains valid for $\lambda_0 = 1 + \tfrac 3 \gamma$ without the additional assumption \eqref{zusatz}. However, as $\tfrac 5 3 + \tfrac 1 \gamma < 1 + \frac{3}{\gamma}$ for $\gamma < 3$ ($\frac{13}{6} \approx 2.17$ instead of $2.5$ for $\gamma = 2$), the additional assumption \eqref{zusatz} allows to lower $\lambda_0$ in this case.

    \item  In view of \eqref{diff.H}, the first term on the right-hand side of  \eqref{est.thm} can be understood as a strengthened version of  $\|\rho_{\eps 0} - \rho_0\|_{L^2(\Oe)}$, and in fact our proof shows that we could also replace the left-hand side term $\|\rho_\eps - \rho\|_{L^\infty(0,T;L^2(\Oe))}^2$ correspondingly in terms of $H$.
    \item Existence of a strong solution to the target system \eqref{limit.system} can be reduced to existence of strong solutions to the porous medium type equation \eqref{FP}, since $\vu$ is then obtained by solving a linear equation. Regarding \eqref{FP}, existence in the case $\vc f = 0$ is standard, see \cite[Proposition~7.21 and Chapter~8.2]{Vazquez2007}. Due to the quadratic lower order term $\dv (\rho^2 \mathcal{K} \vc f)$, the existence of global strong solutions for $\vc f \neq 0$  has not been addressed much in the literature. For attempts in this direction, see \cite{DiazKersner1987, Gilding1989} for the one-dimensional case, and \cite{Watanabe1995} for a multi-dimensional study.
    
  \item We remark that the right-hand side of \eqref{est.thm} can be equivalently written as $E_\eps(\rho_{\eps0},\vu_{\eps0}|\rho_0,\vu(0,\cdot)) + \eps^{\beta}$ (see Section \ref{sec:relativeEnergy} for the definition of $E_\eps$). Indeed, since $\|\rho_{\eps 0}\|_{L^1(\Oe)} \lesssim 1$ due to \eqref{est.initial}, we have
  \begin{align}
   \qquad \quad    \eps^\lambda \|\rho_{\eps 0} |\vu_{\eps 0} - \vu(0,\cdot)|^2\|_{L^1(\Oe)} +  \eps^{\beta}  \lesssim   \eps^\lambda \|\rho_{\eps 0} |\vu_{\eps 0}|^2\|_{L^1(\Oe)} + \eps^{\beta}  \lesssim \eps^\lambda \|\rho_{\eps 0} |\vu_{\eps 0} - \vu(0,\cdot)|^2\|_{L^1(\Oe)} + \eps^{\beta}.
  \end{align}
  It is possible to choose $\rho_{\eps0} = \rho_0$ and $\vu_{\eps0} = \vu(0,\cdot)$ such that the error on the right-hand side of \eqref{est.thm} is just $\eps^{\beta}$.

  \item Darcy's law in the limit system \eqref{limit.system} is (quasi-)stationary. The fluid velocity  $\vu$ is instantaneously determined through $\rho$ and $\vc f$. As expected, \eqref{est.thm} thus ensures convergence of the fluid velocity (in $L_t^2 W_x^{-1,2}$) without requiring convergence of the initial datum. The term $\eps^\lambda \|\rho_{\eps 0} |\vu_{\eps 0}|^2\|_{L^1(\Oe)}$ on the right-hand side of \eqref{est.thm} is small unless $\vu_{\eps 0}$ is too ill-prepared.  
 This term can be  related to an initial layer of thickness  $\eps^\lambda$. We emphasize that in \cite{Masmoudi02} (where $\lambda=4$)  weak convergence $\vu_\eps \wto \vu$ in $L^2((0,T) \times \Omega)$ is shown even when $\eps^\lambda \|\rho_{\eps 0} |\vu_{\eps 0}|^2\|_{L^1(\Oe)}$ is only bounded. Due to the initial layer, we cannot expect strong convergence, though. This is reminiscent of the weak and strong convergence results \cite[Theorem~1.3 and Theorem~1.4]{Hoefer2022} for a related homogenization problem.
    
        \item Note that strong convergence $\vu_\eps \to \vu$ in a nonnegative spatial Sobolev norm cannot be expected since $\vu_\eps = 0$ but $\vu \neq 0$ in $\Omega \setminus \Omega_\eps $.\footnote{Only comparing the $L^2$ norm on $\Omega_\eps$ doesn't overcome this problem due to rapid oscillations of $\vu_\eps$ on the $\eps$-scale.}
        However, we will prove the following $L^2$-estimate when comparing with the modified limit velocity field $\vv_\eps$ from \eqref{r.w.formal}:
        \begin{align*}
            \|\vu_\eps -\vv_\eps\|_{L^2((0,T) \times \Oe)}^2 \lesssim  \int_\Oe H(\rho_{\eps 0}) - H'(\rho_0)(\rho_{\eps 0} - \rho_0) - H(\rho_0) \dd x +  \eps^\lambda \|\rho_{\eps 0} |\vu_{\eps 0}|^2\|_{L^1(\Oe)} + \eps^{\beta}.
        \end{align*}

        \item In view of the formal two-scale expansion, one could expect that an optimal error estimate should hold with $\beta = 2$.
    We are not able to prove such an estimate, though, even if $\Omega = \mathbb T^3$ and for arbitrary large values of $\lambda$ and $\gamma$. This is related to lack of regularity of the pressure: We face errors of the form (see $I_3$ and $I_5$ in Section \ref{sec:estimate.errors})
        $\langle \mathfrak{R}_\eps,r_\eps - \rho_\eps \rangle $ or $\langle \mathfrak{R}_\eps,p(r_\eps) - p(\rho_\eps) \rangle $, where $\mathfrak{R}_\eps$ is a  term of order $O(\eps)$ but only in a negative Sobolev norm. Thus, we would need bounds of the form $|\nabla (r_\eps - \rho_\eps)| \lesssim  \eps$ to deduce a quadratic error in $\eps$. Such bounds seem to be out of reach with the current methods. 

        It is well-known for related homogenization problems (see in particular \cite{Shen2022}) that due to a boundary layer the convergence rate is slower in bounded domains than on the torus. More precisely, one would expect that \eqref{est.thm} with $\beta = 1$
        is optimal in bounded domains. Since the aforementioned lack of regularity and the boundary layer combine in a multiplicative way, we only obtain \eqref{est.thm} with $\beta = \min\{ \tfrac 1 2, 2\lambda-2, \lambda - \frac3\gamma \}$ in bounded domains.
    \end{enumerate}
\end{rem}

\section{Uniform a priori estimates}\label{sec:bds}
\subsection{Uniform bounds for the density and velocity}

\begin{prop}
    Under the assumptions of Theorem \ref{thm1}, we have
 \begin{align}\label{bds}
    \eps^\lambda \|\rho_\eps |\vu_\eps|^2\|_{L^\infty(0,T;L^1(\Oe))} + \|\vu_\eps\|_{L^2((0,T)\times \Oe)}^2 + \eps^2 \|\nabla \vu_\eps\|_{L^2((0,T)\times \Oe)}^2 + \|\rho_\eps\|_{L^\infty(0,T;L^\gamma(\Oe))}^\gamma \lesssim 1.
\end{align}

\end{prop}
\begin{proof}
First, we recall the Poincar\'e inequality in the perforated domain, that for all $q \in [1,\infty]$
\begin{align}\label{Poinc}
    \|\varphi\|_{L^q(\Oe)} \lesssim \eps \|\nabla \varphi\|_{L^q(\Oe)} \quad \text{for all } \varphi \in W^{1,q}(\Omega) \text{ with } \varphi = 0 \text{ in } \Omega \setminus \Oe , 
\end{align}
which is standard, see \cite[Lemma 1]{Tartar80}.

To derive \emph{a priori} uniform bounds for the density $\rho_\eps$ and the velocity $\vu_\eps$, we use the energy inequality \eqref{EnergIneTime} and mass conservation, i.e., for almost every $\tau \in (0,T)$,
\begin{align*}
    E_\eps(\tau) + \int_0^\tau \int_\Oe \eps^2 \S(\nabla \vu_\eps) : \nabla\vu_\eps \dd x \dd t \leq E_\eps(0) + \int_0^\tau \int_\Oe \rho\vc{f}\cdot\vu_\eps \dd x,
\end{align*}
where we set $E_\eps(\tau) = E_\eps \left( \rho_\eps, \vu_\eps \Big| \fint_{\Oe} \rho_\eps \dd x, 0 \right)(\tau)$. The force term is estimated as
\begin{align*}
    \int_\Oe |\rho_\eps \vc{f} \cdot \vu_\eps| \dd x &\leq \|\rho_\eps\|_{L^2(\Oe)} \|\vu_\eps\|_{L^2(\Oe)} \leq C\|\rho_\eps\|_{L^2(\Oe)} \eps \|\nabla\vu_\eps\|_{L^2(\Oe)}\\
    &\leq C \|\rho_\eps\|_{L^2(\Oe)}^2 + \frac{1}{2C_K} \eps^2 \|\nabla\vu_\eps\|_{L^2(\Oe)}^2 \leq C \|\rho_\eps\|_{L^2(\Oe)}^2 + \int_\Oe \frac12 \eps^2 \S(\nabla\vu_\eps) : \nabla\vu_\eps \dd x,
\end{align*}
where we used inequality \eqref{Poinc}, and $C_K>0$ is the constant in Korn's inequality on $\Omega$, which can be used after extending $\vu_\eps$ by zero to $\Omega$. Finally, since $\gamma\geq 2$, the properties of $p$ (see \eqref{p}) and the definition and monotonicity of $H$ (see \eqref{defPressPot}) imply $\|\rho_\eps\|_{L^2(\Oe)}^2 \lesssim 1+\|\rho_\eps\|_{L^\gamma(\Oe)}^\gamma\lesssim 1+\|p(\rho_\eps)\|_{L^1(\Oe)} \lesssim 1+\|H(\rho_\eps)\|_{L^1(\Oe)}$, hence
\begin{align*}
    E_\eps(\tau) + \int_0^\tau \int_\Oe \frac12 \eps^2 \S(\nabla \vu_\eps) : \nabla\vu_\eps \dd x \dd t \lesssim 1 + E_\eps(0) + \int_0^\tau E_\eps(t) \dd t.
\end{align*}
Thus, applying Gr\"onwall's inequality, \eqref{Poinc} and \eqref{est.initial}, we conclude \eqref{bds}.
\end{proof}

\subsection{Improved pressure estimates}
Following the lines of \cite[Section~3.1]{Masmoudi02} and \cite[Section~3.3]{OschmannPokorny2023}, we extend the pressure to all of $\Omega$ and then split the extended pressure in a regular part and a small error. 

\begin{lem} \label{lem:pressure.decomp}
    There exists a pressure extension $P_\eps : \Omega \to [0,\infty)$ such that $P_\eps = p(\rho_\eps)$ in $\Omega_\eps$ and which is uniformly bounded in $L^\infty(0,T;L^1(\Omega))$. 
    Moreover, for $p \in [\frac {2\gamma}{\gamma-1},6]$
  \begin{align*}
    P_\eps \in & \,L^2(0,s;W^{1,2}( \Omega)) + \eps L^2((0,s)\times \Omega) + \eps^{\lambda - \frac{3(p-2)}{p}} L^1(0,s;L^{\frac{p \gamma}{p + 2\gamma}}(\Omega)) + \eps^{\lambda -1 - \frac{3(p-2)}{p}} L^1(0,s;W^{1,\frac{p \gamma}{p + 2\gamma}}(\Omega))   \\
      & + \eps^{\lambda} [W^{1,2}(0,s; L^{\frac{ 6\gamma}{5\gamma-6}}(\Omega))]' + \eps^{\lambda-1} [W^{1,2}(0,s; W^{-1,\frac{ 6\gamma}{5\gamma-6}}(\Omega))]'
     \\
      & + \eps^{\frac{\lambda}{2}+1} [W^{1,\infty}(0,s; L^{\frac{2\gamma}{\gamma-1}}(\Omega))]' + \eps^\frac{\lambda}{2} [W^{1,\infty}(0,s; W^{-1,\frac{2\gamma}{\gamma-1}}(\Omega))]',
\end{align*}  
with a uniform bound in $\eps$ in this space. In particular, for all $p \in [\frac {2\gamma}{\gamma-1},6]$ and all $\varphi \in W^{1,\infty}(0,s;L^\infty(\Omega))$,  
    \begin{align*}
       \int_0^s \int_{\Oe} p(\rho_\eps) \varphi \dd x \dd t \lesssim \|\varphi\|_{W^{1,\infty}(0,s;L_{\eps}^{2})} &+\eps^{\lambda - 1 - \frac{3(p-2)}{p}} \|\varphi\|_{W^{1,\infty}(0,s;L_{\eps}^{\frac{p\gamma}{(p-2)\gamma -p}})} +    \eps^{\lambda-1}\|\varphi\|_{W^{1,\infty}(0,s;L_{\eps}^{\frac{ 6\gamma}{5\gamma-6}})} \\
        &+ \eps^\frac{\lambda}{2}\|\varphi\|_{W^{1,\infty}(0,s;L_{\eps}^\frac{2\gamma}{\gamma-1})},  
    \end{align*}
    where 
    \begin{align}\label{i2}
       \|\cdot\|_{L^{q}_\eps} := \eps\|\cdot\|_{L^q(\Omega)} + \|\cdot\|_{[W^{1,q'}(\Omega)]'}.
    \end{align}
\end{lem}

For the proof, we recall the following result from \cite[Lemma~2.2]{Mikelic91}.
\begin{prop}\label{prop:restOp}
Let $1 < q < \infty$. There exists an operator
\begin{align*}
    R_\eps : W_0^{1,q}(\Omega)\to W_0^{1,q}(\Omega_\eps)
\end{align*}
such that:
\begin{itemize}
    \item For any $\vu\in W_0^{1,q}(\Omega_\eps)$, we have $R_\eps \tilde{\vu}=\vu$.
    \item $\dv \vu = 0 \Rightarrow \dv R_\eps \vu = 0$.
    \item There exists a constant $C>0$ which is independent of $\eps$ such that
    \begin{align}\label{est.R_eps}
        \eps^{-1} \|R_\eps \vu\|_{L^q(\Oe)} + \|\nabla R_\eps \vu\|_{L^q(\Oe)} \leq C \big( \eps^{-1} \|\vu\|_{L^q(\Omega)} + \|\nabla \vu\|_{L^q(\Omega)} \big).
    \end{align}
\end{itemize}
\end{prop}

\begin{proof}[Proof of Lemma \ref{lem:pressure.decomp}]
With the help of the operator $R_\eps$, we define the extension $P_\eps=P(p(\rho_\eps))$ of the pressure $p(\rho_\eps)$ to the whole of $\Omega$ by duality as
\begin{align*}
    \int_0^s \langle \nabla P_\eps, \phi\rangle_{W^{-1,2}(\Omega); W_0^{1,2}(\Omega)} \dd t = \int_0^s \langle \nabla p(\rho_\eps), R_\eps \phi\rangle_{W^{-1,2}(\Omega_\eps); W_0^{1,2}(\Omega_\eps)} \dd t \quad \text{for any } \phi\in C_c^\infty([0,s]\times \Omega).
\end{align*}
The fact that the operator $\nabla P_\eps$ is indeed a gradient follows from the second property of the operator $R_\eps$.
By the second equation of \eqref{NSE}, we have
\begin{align*}
    \int_0^s \langle \nabla p(\rho_\eps), R_\eps \phi\rangle_{W^{-1,2}(\Omega_\eps); W_0^{1,2}(\Omega_\eps)} \dd t &= \int_0^s \int_\Oe \eps^\lambda \rho_\eps \vu_\eps \cdot \partial_t R_\eps \phi \dd x \dd t + \int_0^s \int_\Oe \eps^\lambda \rho_\eps \vu_\eps\otimes\vu_\eps : \nabla R_\eps \phi \dd x \dd t\\
    &\quad - \int_0^s\int_\Oe \eps^2 \S(\nabla\vu_\eps) : \nabla R_\eps \phi \dd x \dd t + \int_0^s \int_\Oe \rho_\eps \vc{f} \cdot R_\eps\phi \dd x \dd t\\
    &\quad + \int_\Oe \eps^\lambda \vc m_{\eps0} R_\eps \phi(0,\cdot) -  \eps^\lambda\rho_\eps(s,\cdot) \vu_\eps(s,\cdot) R_\eps \phi(s,\cdot) \dd x\\
    &=\sum_{i=1}^5 I_i.
\end{align*}

We will estimate each term separately. We get by  the Sobolev embedding $W_0^{1,2} \hookrightarrow L^6$ and the uniform bounds \eqref{bds} and  \eqref{est.R_eps}
\begin{align*}
    |I_1| &\ls  \eps^\lambda \|\rho_\eps\|_{L^\infty(0,s;L^\gamma(\Oe))} \|\vu_\eps\|_{L^2(0,s;L^{6}(\Oe))} \|\partial_t R_\eps\phi\|_{L^2(0,s; L^{\frac{6\gamma}{5 \gamma -6}}(\Oe))}\\
    &\ls \eps^{\lambda -1} \Big[ \|\partial_t \phi\|_{L^2(0,s; L^{\frac{6\gamma}{5 \gamma -6}}(\Omega))} + \eps \|\partial_t \nabla \phi\|_{L^2(0,s; L^{\frac{6\gamma}{5 \gamma -6}}(\Omega))} \Big]\\
    &\ls \eps^{\lambda -1} \Big[ \|\phi\|_{W^{1,2}(0,s;  L^{\frac{6\gamma}{5 \gamma -6}}(\Omega))} + \eps \|\nabla \phi\|_{W^{1,2}(0,s;  L^{\frac{6\gamma}{5 \gamma -6}}(\Omega))} \Big].
\end{align*}
For $I_2$, we use that by H\"older's inequality, Sobolev embedding,   the Poincar\'e inequality \eqref{Poinc}, and \eqref{bds} for any $p \in [2,6]$
\begin{align} \label{interpolation}
    \|\vu_\eps\|_{L^p(\Omega_\eps)} \lesssim   \|\vu_\eps\|_{L^2(\Oe)}^{\frac{6 -p}{2p}} \|\vu_\eps\|_{L^{6}(\Oe)}^{\frac{3(p-2)}{2p}}  \lesssim \eps^{\frac{6 -p}{2p}} \|\nabla \vu_\eps\|_{L^2(\Oe)} \lesssim \eps^{\frac{6 -3p}{2p}} .
\end{align}
Hence, for any $p \in [2,6]$ such that $p \geq \frac {2\gamma}{\gamma-1}$,
\begin{align*}
    |I_2| &\ls \eps^\lambda \|\rho_\eps\|_{L^\infty(0,s;L^\gamma(\Oe))} \| \vu_\eps\|_{L^2(0,s; L^p(\Oe))}^2 \|\nabla R_\eps\phi\|_{L^\infty(0,s;L^\frac{p \gamma}{(p-2)\gamma-p}(\Oe))}\\
    &\ls \eps^{\lambda -\frac{3(p-2)}{p}} \Big[ \|\nabla\phi\|_{L^\infty(0,s;L^\frac{p \gamma}{(p-2)\gamma-p}(\Omega))} + \eps^{-1}\|\phi\|_{L^\infty(0,s;L^\frac{p \gamma}{(p-2)\gamma-p}(\Omega))} \Big],\\
    |I_3| &\ls \eps^2 \|\S(\nabla\vu_\eps)\|_{L^2((0,s)\times \Oe)} \|\nabla R_\eps\phi\|_{L^2((0,s)\times \Oe)}\\
    &\ls \|\phi\|_{L^2((0,s)\times \Omega)} + \eps \|\nabla\phi\|_{L^2((0,T)\times\Omega)},\\
    |I_4| &\ls \|\rho_\eps\|_{L^2((0,T)\times \Oe)} \|R_\eps\phi\|_{L^2((0,s)\times \Oe)} \ls \|\phi\|_{L^2((0,s)\times \Omega)} + \eps \|\nabla\phi\|_{L^2((0,s)\times \Omega)}.
\end{align*}

For $I_5$, we use that $W^{1,\infty}(0,s) \subset C^0(0,s)$, uniformly for all $s > 0$ and $W_0^{1,2}(\Omega)\subset L^\frac{2\gamma}{\gamma-1}(\Omega)$ for any $\gamma>\frac32$ to estimate
using again \eqref{bds}
\begin{align*}
    |I_5| &\ls \eps^\lambda \|\rho_\eps \vu_\eps\|_{L^\infty(0,s; L^\frac{2\gamma}{\gamma+1}(\Oe))} \|R_\eps \phi\|_{L^\infty(0,s;L^\frac{2\gamma}{\gamma-1}(\Oe))} \\
    &\ls \eps^\lambda \|\rho_{\eps} |\vu_{\eps}|^2\|_{L^\infty(0,s;L^1(\Oe))}^\frac12 \|\rho_{\eps}\|_{L^\infty(0,s;L^\gamma(\Oe))}^\frac12 \|R_\eps \phi\|_{L^\infty(0,s;L^\frac{2\gamma}{\gamma-1}(\Oe))} \\
    &\ls
    \eps^\frac{\lambda}{2}  \Big[ \|\phi\|_{W^{1,\infty}(0,s; L^\frac{2\gamma}{\gamma-1}(\Omega))} + \eps \|\nabla \phi\|_{W^{1,\infty}(0,s; L^\frac{2\gamma}{\gamma-1}(\Omega))} \Big] .
\end{align*}
Combining these estimates on $I_k$ for  $1 \leq k \leq 5$ yields
\begin{align*}
    \nabla P_\eps \in & \,L^2((0,s)\times \Omega) + \eps L^2(0,s;W^{-1,2}(\Omega))  + \eps^{\lambda - \frac{3(p-2)}{p}} L^1(0,s;W^{-1,\frac{p \gamma}{p + 2\gamma}}(\Omega)) + \eps^{\lambda -1 - \frac{3(p-2)}{p}} L^1(0,s;L^{\frac{p \gamma}{p + 2\gamma}}(\Omega))   \\
     & + \eps^{\lambda} [W^{1,2}(0,s; W^{1,\frac{6\gamma}{5\gamma -6}}(\Omega))]' + \eps^{\lambda-1} [W^{1,\infty}(0,s; L^{\frac{6\gamma}{\gamma +6}}(\Omega))]'
    \\
     & + \eps^{\frac{\lambda}{2}+1} [W^{1,\infty}(0,s; W^{1,\frac{2\gamma}{\gamma-1}}(\Omega))]' + \eps^\frac{\lambda}{2} [W^{1,\infty}(0,s; L^\frac{2\gamma}{\gamma-1}(\Omega))]'.
\end{align*}

Note that regarding the estimate of $I_2$, we can test with $\phi \in L^\infty(0,s;W^{1,\tfrac{p \gamma}{(p-2)\gamma-p}}(\Omega))$ such that the corresponding part of $\nabla P_\eps$ is indeed an element of $L^1(0,s;W^{-1,\frac{p \gamma}{p + 2\gamma}}(\Omega))$.

The above decomposition for $\nabla P_\eps$ now implies
\begin{align*}
    P_\eps \in & \,L^2(0,s;W^{1,2}( \Omega)) + \eps L^2((0,s)\times \Omega) + \eps^{\lambda - \frac{3(p-2)}{p}} L^1(0,s;L^{\frac{p \gamma}{p + 2\gamma}}(\Omega)) + \eps^{\lambda -1 - \frac{3(p-2)}{p}} L^1(0,s;W^{1,\frac{p \gamma}{p + 2\gamma}}(\Omega))   \\
      & + \eps^{\lambda} [W^{1,2}(0,s; L^{\frac{6\gamma}{5\gamma -6}}(\Omega))]' + \eps^{\lambda-1} [W^{1,2}(0,s; W^{-1,\frac{6\gamma}{5\gamma - 6}}(\Omega))]'
     \\
      & + \eps^{\frac{\lambda}{2}+1} [W^{1,\infty}(0,s; L^{\frac{2\gamma}{\gamma-1}}(\Omega))]' + \eps^\frac{\lambda}{2} [W^{1,\infty}(0,s; W^{-1,\frac{2\gamma}{\gamma-1}}(\Omega))]',
\end{align*}
which is the desired decomposition for $P_\eps$. 

We note moreover that due to \cite[Theorem 2.6]{Allaire89} (see also \cite{LiptonAvellaneda1990}), we have the explicit representation
\begin{align*}
    P_\eps = \begin{cases}
    p(\rho_\eps) & \mbox{in } \Omega_\eps,\\
    \fint_{Q_i^\eps \setminus \obst_i^\eps} p(\rho_\eps) \dd x &\mbox{in } \obst_i^\eps,
    \end{cases}
\end{align*}
from which we also infer that $P_\eps$ is uniformly bounded in $L^\infty(0,T;L^1(\Omega))$. In particular, the ``distributional'' parts of $P_\eps$ are indeed functions rather than mere distributions.
\end{proof}

\section{Proof of the main result}\label{sec:proof}
This section is devoted to the proof of Theorem \ref{thm1} in the case that $\Omega$ is the torus. Consequently, we will assume $\Omega = \mathbb T^3$ throughout this section. We start with some preparation before giving the explicit error estimates.\\

\subsection{Preliminary estimates on the test functions and the pressure potential}
\label{sec:correctors}
As outlined in Section \ref{sec:formal}, we  use the relative energy inequality with  $(r_\eps,\vv_\eps)$ as in \eqref{r.w.formal}. 
To simplify the notation, we recall the definition of $W = (\vv_1,\vv_2,\vv_3)$ and $\vc q = (q_1,q_2,q_3)^T$ from \eqref{sol.Stokes} and define $(W^\eps, \vc{q}^\eps) \in H^1(\Omega) \times L^2(\Oe)/\mathbb{R}$ via
\begin{align*}
    W^\eps(x) &= W(x/\eps)\mathcal{K}^{-1} = \big[ (\vv_1, \vv_2, \vv_3)(x/\eps) \big] \mathcal{K}^{-1}, \\
    \vc{q}^\eps(x) &= \eps^{-1} \mathcal K^{-1} \vc{q}(x/\eps) =  \mathcal{K}^{-1}  \big[ \eps^{-1} (q_1, q_2, q_3)(x/\eps) \big]^T,
\end{align*}
where we recall that $\mathcal K$ defined in \eqref{def:K} is a symmetric positive definite matrix.
 Note that by this definition and our convention for the divergence and gradient (see Section~\ref{sec:structure}), we have
\begin{align}\label{definition.K}
\begin{cases}
     \eps^2(-\Delta W^\eps + \nabla \vc{q}^\eps) = \mathcal{K}^{-1} & \text{in } \Oe,\\
     \dv W^\eps = 0 & \text{in } \Oe,\\
     W^\eps = 0 & \text{in } \Omega\setminus \Oe.
\end{cases}
\end{align}
Moreover, \eqref{def:K} implies
\begin{align} \label{average.W}
        \fint_{Q_i^\eps \setminus \obst_i^\eps} W^\eps \dd x = \frac 1 \theta \rm Id,
\end{align}
where we recall $\theta = 1 - |\obst| / |Q|$ from \eqref{theta}.
Since $\vc q_\eps$ is just defined up to an additive constant, without loss of generality we fix this constant such that $\int_\Oe \vc q_\eps \dd x = 0$.

Using that $\mathcal K$ is symmetric, \eqref{r.w.formal} can now be rewritten as
\begin{align}\label{cor}
\begin{split}
    r_\eps &=  p^{-1}(p(\rho) + \eps^2 \vc q^\eps \cdot \vu), \\
    \vv_\eps &= W^\eps \vu.
\end{split}
\end{align}

\begin{lem} \label{lem:r.v}
Under the assumptions of Theorem \ref{thm1}, there exists $\eps_0 > 0$ 
such that for  all $\eps < \eps_0$, the functions $r_\eps, \vv_\eps \in W^{1,\infty}((0,T) \times \Omega)$ are well-defined and satisfy
\begin{align}
    1 \lesssim r_\eps &\lesssim 1,  \label{r_eps}\\
    \|\vv_\eps\|_{W^{1,\infty}(0,T; L^\infty(\Omega))} +  \|\dv \vv_\eps\|_{L^\infty((0,T) \times \Omega)} + \| r_\eps\|_{W^{1,\infty}((0,T) \times \Omega)}  &\lesssim 1, \label{v_eps} \\
    \|\nabla \vv_\eps\|_{L^\infty((0,T) \times \Omega)} &\lesssim \eps^{-1}, \label{nabla.v_eps} \\
  \|p(r_\eps) - p(\rho)\|_{L^\infty((0,T) \times \Omega)} +   \|r_\eps - \rho\|_{L^\infty((0,T) \times \Omega)} + \|\partial_t (r_\eps - \rho)\|_{L^\infty((0,T) \times \Omega)}  &\lesssim \eps. \label{rho_cor}
\end{align}

Furthermore,
\begin{align} \label{W^eps.H^-1}
    \bigg\|\bigg( \frac 1 \theta \Id - W^\eps \bigg) \1_\Oe \bigg\|_{[W^{1,1}(\Omega)]'} \ls \eps.
\end{align}

Finally, we also have
\begin{align} \label{W^eps.H^-1.Omega}
    \|\Id - W^\eps \|_{[W^{1,1}(\Omega)]'} \ls \eps.
\end{align}
\end{lem}

\begin{proof}
First, standard regularity theory for Stokes equations implies
\begin{align*}
    \|\vv_i\|_{W^{1,\infty}(Q)} + \|q_i\|_{W^{1,\infty}(Q\setminus \obst)} \lesssim 1.
\end{align*}
Hence, since $\mathcal{K}$ is constant,
\begin{align} \label{W.q.infty}
    \|W^\eps\|_{L^\infty(\Omega)} + \eps \|\nabla W^\eps\|_{L^\infty(\Omega)} +  \eps \|\vc{q}^\eps\|_{L^\infty(\Oe)} + \eps^2  \|\nabla \vc{q}^\eps\|_{L^\infty(\Oe)}\lesssim  1.
\end{align}
In particular 
$ \|\eps^2 \vc q^\eps \cdot \vu\|_{L^\infty((0,T) \times \Omega)} \lesssim \eps$.
This implies that $r_\eps$ is well-defined for $\eps > 0$ sufficiently small and satisfies \eqref{r_eps}. Indeed, the function $p$ is invertible on $(0,\infty)$ by \eqref{p}. 
Moreover, the assumption $\rho \geq \underline{\rho} >0$
in the statement of Theorem \ref{thm1} together with the estimate \eqref{W.q.infty} guarantees that the argument of $p^{-1}$ in \eqref{cor} is bounded above by a positive constant for $\eps \ll 1$.
Moreover, $\vv_\eps$ is well defined and the estimates \eqref{v_eps}--\eqref{nabla.v_eps} hold due to \eqref{W.q.infty} and using also that $\dv W^\eps = 0$. Finally, using \eqref{r_eps}, the regularity of $p$, and $ \|\eps^2 \vc q^\eps \cdot \vu\|_{L^\infty((0,T) \times \Omega)} \lesssim \eps$, estimate \eqref{rho_cor} follows.

\medskip

To show \eqref{W^eps.H^-1},  we observe that by \eqref{average.W}, we have for all cells $Q_i^\eps \subset \Omega$ 
\begin{align*}
    \int_{Q_i^\eps \setminus \obst_i^\eps} \frac 1 \theta \Id - W^\eps\dd x = 0.
\end{align*}
Hence, for $\Psi \in W^{1,1}(\Omega)$, we have $\Psi \cdot \1_\Oe \in W^{1,1}(\Oe)$; thus, using as well $\Omega_\eps = \bigcup_{i=1}^{N(\eps)} Q_i^\eps \setminus \obst_i^\eps$ if $\Omega = \mathbb T^3$,
\begin{align*}
    \int_{\bigcup_{i=1}^{N(\eps)} Q_i^\eps \setminus \obst_i^\eps} \bigg(\frac 1 \theta \Id - W^\eps \bigg) : \Psi \dd x &= \sum_{i=1}^{N(\eps)} \int_{Q_i^\eps \setminus \obst_i^\eps} \bigg(\frac 1 \theta \Id - W^\eps \bigg) : \bigg(\Psi - \fint_{Q_i^\eps \setminus \obst_i^\eps} \Psi \dd y \bigg) \dd x \\
    &\lesssim  \eps \sum_{i=1}^{N(\eps)} \|\nabla \Psi\|_{L^1(Q_i^\eps \setminus \obst_i^\eps)} \leq \eps  \|\nabla \Psi\|_{L^1(\Omega)}.
\end{align*}

Finally, we observe that we can redo this computation with $\Oe$ replaced by $\Omega$ and dropping the factor $\theta^{-1}$ since 
\eqref{def:K} implies that for all cells $Q_i^\eps \subset \Omega$ 
\begin{align*}
    \int_{Q_i^\eps }  \Id - W^\eps\dd x = 0.
\end{align*}
This yields \eqref{W^eps.H^-1.Omega}.
\end{proof}

For the readers convenience, we recall some well known estimates (see, e.g., \cite[Lemma~5.1]{FeireislNovotny2009singlim}).

\begin{lem}  \label{lem:diff.H}
   Assume that $p$ satisfies \eqref{p} for some $\gamma \geq 2$. Let $I \Subset (0,\infty)$. Then, uniformly for all $r \in I$ and all $s \in [0,\infty)$,
\begin{align} \label{diff.H}
    H(s) - H'(r)(s-r) - H(r) &\gtrsim (r-s)^2 + p(s) \1_{\{s \geq 2r\}} \gtrsim |p(s) - p'(r)(s - r) - p(r) |.
\end{align}
\end{lem}
\begin{proof}
We only give the proof of the first inequality in \eqref{diff.H}. The second one follows by similar arguments, recalling that $p \in C^2((0,\infty))$.

We remark that by the definition of $H$ in \eqref{defPressPot}, we have
$H''(s) = p'(s)/s$ and hence
\begin{align*}
    H(s) - H'(r)(s-r) - H(r) = \int^r_s (z-s) H''(z) \dd z = \int_s^r \frac{z-s}z p'(z) \dd z.
\end{align*}
By the assumptions on $p$ in \eqref{p} 
and $\gamma \geq 2$ we have for $s \geq r/2$ that $p'(z)/z \gtrsim z^{\gamma-2} \gtrsim 1$ such that
\begin{align*}
    H(s)-H(r)-H'(r)(s-r) \gtrsim (s-r)^2.
\end{align*}

On the other hand, if $s \leq r/2$ we just use $r/2 \leq (r-s) \leq r$ to get 
\begin{align*}
    H(s) - H'(r)(s-r) - H(r) \geq \int_{r/2}^r \frac{z-s}z p'(z) \dd z \gtrsim \int_{r/2}^r (z-s) \dd z \geq r^2 \gtrsim (r -s)^2.
\end{align*}

It remains to show that for $s \geq 2r$ we have
\begin{align} \label{H.p}
     H(s) - H'(r)(s-r) - H(r) &\gtrsim  p(s).
\end{align}
By \eqref{p} there exists $z_0 \geq  r$  such that $\tfrac 1 2 p_\infty z^{\gamma-1} \leq p'(z) \leq  2 p_\infty z^{\gamma-1}$ for all $z \geq z_0$. Hence, for $s \geq 2 z_0$, we have
\begin{align*}
    H(s) - H'(r)(s-r) - H(r) = \int_r^s \frac{s-z}{z} p'(z) \dd z &\geq \frac{p_\infty}{2} \int_{z_0}^s \frac{s-z}{z} z^{\gamma-1} \dd z \\
    &= \frac {p_\infty} 2 \left( \left(\frac 1 {\gamma - 1} - \frac 1 \gamma \right) s^\gamma - \left(\frac{s}{\gamma -1} - \frac{z_0}{\gamma} \right) z_0^{\gamma -1}\right) \gtrsim s^{\gamma}.
\end{align*}
On the other hand,
\begin{align*}
    p(s) = p(z_0) + \int_{z_0}^s p'(z) \dd z \leq p(z_0) + \frac{2 p_\infty} {\gamma} s^\gamma \lesssim s^\gamma.
\end{align*}
This yields \eqref{H.p} for $s \geq 2 z_0$.
For $2r \leq s \leq 2 z_0$, we simply estimate 
\begin{align*}
   p(s) \lesssim 1 \lesssim \int_r^{2r} \frac{2r -z}{z} p'(z) \dd z \leq \int_r^{s} \frac{s -z}{z} p'(z) \dd z  =  H(s) - H'(r)(s-r) - H(r).
\end{align*}
\end{proof}

\subsection{Manipulation of the relative energy inequality} \label{sec:manip.rel.energy}

By Remark \ref{rem:rel.energy} and Lemma \ref{lem:r.v} we can use $(r_\eps, \vv_\eps)$, defined in \eqref{cor}, in the relative energy inequality \eqref{relen}. With the short notation
\begin{align*}
    E_\eps(\tau) = E_\eps(\rho_\eps, \vu_\eps| r_\eps, \vv_\eps)(\tau),
\end{align*}
we have
\begin{align}\label{relen1}
    E_\eps(\tau) + \int_0^\tau \int_\Oe \eps^2 \big( \mathbb{S}(\nabla \vu_\eps) - \mathbb{S}(\nabla \vv_\eps) \big) : (\nabla\vu_\eps-\nabla\vv_\eps \big) \dd x \dd t \leq E_\eps(0) + \int_0^\tau \mathcal{R}_\eps \dd t,
\end{align}
where the remainder $\mathcal{R}_\eps$ is given by
\begin{align*}
    \mathcal{R}_\eps &= \eps^\lambda\int_\Oe  \rho_\eps (\partial_t \vv_\eps + (\vu_\eps \cdot \nabla) \vv_\eps) \cdot (\vv_\eps - \vu_\eps) \dd x + \eps^2 \int_\Oe  \mathbb{S}(\nabla\vv_\eps) : \nabla (\vv_\eps - \vu_\eps) \dd x \\
    &\quad + \int_\Oe \rho_\eps \vc{f} \cdot (\vu_\eps - \vv_\eps) \dd x + \int_\Oe (r_\eps-\rho_\eps)\partial_t H'(r_\eps) + \nabla H'(r_\eps)\cdot (r_\eps\vv_\eps - \rho_\eps\vu_\eps) \dd x \\
    &\quad - \int_\Oe \dv\vv_\eps (p(\rho_\eps)-p(r_\eps)) \dd x =: \sum_{i=1}^5 \mathcal{R}^i_\eps.
\end{align*}
We start by rewriting the second term on the right-hand side as
\begin{align*}
    \mathcal{R}^2_\eps &= \eps^2  \int_\Oe \vc (-\Delta W^\eps \vu - \vc z_\eps) \cdot (\vv_\eps - \vu_\eps) \dd x 
\end{align*}
with
\begin{align*}
    (\vc z_\eps)_k = \left( \dv \S \big( \nabla \vu \ (W^\eps)^T \big) \right)_k + \left(\sum_{i,j=1}^3 \partial_i  W^\eps_{kj} \partial_i \vu_j + \partial_k  W^\eps_{ij} \partial_i \vu_j\right),
\end{align*}
where we used $\dv W^\eps = 0$. Therefore, $\vc z_\eps$ satisfies
\begin{align*}
    \|\vc z_\eps\|_{L^2(\Oe)} \lesssim \eps^{-1}.
\end{align*}
Next, we use Darcy's law in \eqref{limit.system} satisfied by $(\rho,\vu)$ as well as $\eps^2(-\Delta W^\eps + \nabla \vc q^\eps) = \mathcal K^{-1}$ from \eqref{definition.K} to deduce
\begin{align*}
    \mathcal{R}^2_\eps &=  \eps^2 \int_\Oe (- \nabla \vc q^\eps \vu - \vc z_\eps + \rho \vc f - \nabla p(\rho)) \cdot (\vv_\eps - \vu_\eps) \dd x.
\end{align*}
Using that by definition of $r_\eps$, we have
\begin{align*}
    p(r_\eps) = p(\rho) + \eps^2 \vc q^\eps \cdot \vu,
\end{align*}
we arrive at
\begin{align*}
        \mathcal{R}^2_\eps &=\int_\Oe (\eps^2 \tilde{\vc  z}_\eps + \rho \vc f) \cdot (\vv_\eps - \vu_\eps)  - \nabla p(r_\eps) \cdot (\vv_\eps - \vu_\eps)\dd x,
\end{align*}
where
\begin{align} \label{tilde.z}
    \tilde{\vc  z}_\eps =  \nabla \vu \cdot  \vc q^\eps - \vc z_\eps , \qquad  \|\tilde{\vc  z}_\eps\|_{L^2(\Oe)} \lesssim \eps^{-1}.
\end{align}

Let us also rearrange the terms  $\mathcal R_\eps^i$, $i=4,5$, in combination with the  term $\nabla p(\rho) \cdot (\vv_\eps - \vu_\eps) $ on the right-hand side of $\mathcal{R}^2_\eps$ above.
We start by observing that for any $s$, the definition of $H$ in \eqref{defPressPot} yields the algebraic relations
\begin{align*}
    \partial_t H'(s) = \frac{p'(s)}{s} \partial_t s, \qquad \nabla H'(s)  = \frac{p'(s)}{s} \nabla s = \frac{\nabla p(s)}{s}.
\end{align*}
Thus,
\begin{align*}
    (r_\eps-\rho_\eps)\partial_t H'(r_\eps) + \nabla H'(r_\eps)\cdot (r_\eps\vv_\eps - \rho_\eps\vu_\eps) - \nabla p(r_\eps) \cdot (\vv_\eps - \vu_\eps)  
    &= \left(1 - \frac{\rho_\eps}{r_\eps}\right) p'(r_\eps)\left(\partial_t r_\eps + \vu_\eps \cdot \nabla r_\eps \right).
\end{align*}
We further compute using the continuity equation of $\rho$ from \eqref{limit.system}
\begin{align} \label{manip.R_4,5}
\begin{aligned}
p'(r_\eps)\left(\partial_t r_\eps + \vu_\eps \cdot \nabla r_\eps \right)
    &= \nabla p(r_\eps) \cdot (\vu_\eps - \vv_\eps) - \frac 1 \theta p'(r_\eps) \rho \dv \vu\\
    &\quad + p'(r_\eps)\bigg(\partial_t (r_\eps -\rho) + (W^\eps - \frac 1 \theta \Id) \vu \cdot \nabla \rho + W^\eps \vu \cdot \nabla(r_\eps - \rho) \bigg).
    \end{aligned}
\end{align}
Hence, we arrive at
\begin{align}\label{errors.final}
\begin{aligned}
     \mathcal{R}_\eps &= \eps^\lambda\int_\Oe  \rho_\eps (\partial_t \vv_\eps + (\vu_\eps \cdot \nabla) \vv_\eps) \cdot (\vv_\eps - \vu_\eps) \dd x\\
     &\quad + \int_\Oe (\rho_\eps- \rho) \vc{f} \cdot (\vu_\eps - \vv_\eps) \dd x\\
     &\quad + \int_\Oe -\left( 1 - \frac{\rho_\eps}{r_\eps} \right) \frac 1 \theta p'(r_\eps) \rho \dv \vu  - \dv\vv_\eps (p(\rho_\eps)-p(r_\eps)) \dd x \\
     &\quad + \int_\Oe \left(1 - \frac{\rho_\eps}{r_\eps} \right) \nabla p(r_\eps) \cdot (\vu_\eps - \vv_\eps) \dd x \\
     &\quad + \int_\Oe \left( 1 - \frac{\rho_\eps}{r_\eps} \right) p'(r_\eps) \bigg(\partial_t (r_\eps -\rho) + \bigg(W^\eps - \frac 1 \theta \Id \bigg) \vu \cdot \nabla \rho + W^\eps \vu \cdot \nabla(r_\eps - \rho) \bigg)\dd x \\
     &\quad + \int_\Oe \eps^2 \tilde{\vc z}_\eps \cdot (\vv_\eps - \vu_\eps) \dd x\\
     &=: \sum_{k=1}^6 I_k.
     \end{aligned}
\end{align}
Our ultimate goal is to estimate the error terms by the relative energy $E_\eps$ plus a small error. This will be carried out in the next section.

\subsection{Estimate of the error terms} \label{sec:estimate.errors}

We can now estimate the terms $I_k$ for any $\delta > 0$ with a suitable $C_\delta < \infty$ as follows.
The easiest is $I_6$ which in view of \eqref{tilde.z} and \eqref{Poinc} is estimated as
\begin{align} \label{I_6}
    |I_6| \lesssim \delta \eps^2 \|\nabla(\vu_\eps - \vv_\eps)\|_{L^2(\Oe)}^2 + C_\delta \eps^2.
\end{align}

For $I_1$, we estimate
\begin{align} \label{I_1.1}
\begin{aligned}
    |I_1|&\leq 
\bigg| \eps^\lambda \int_\Oe \rho_\eps (\partial_t \vv_\eps + (\vv_\eps \cdot \nabla) \vv_\eps) \cdot (\vv_\eps - \vu_\eps) \dd x \bigg| \\
&\quad+ \bigg| \eps^\lambda \int_\Oe \rho_\eps (\partial_t \vv_\eps + ((\vu_\eps-\vv_\eps) \cdot \nabla) \vv_\eps) \cdot (\vv_\eps - \vu_\eps) \dd x \bigg| \\
&\leq  \eps^{\lambda-1} \|\rho_\eps\|_{L^2(\Oe)} \|\vv_\eps -\vu_\eps\|_{L^2(\Oe)} + \eps^{\lambda -1} \|\rho_\eps \|_{L^\gamma(\Oe)} \| \vv_\eps -\vu_\eps \|^2_{L^{\frac {2\gamma}{\gamma -1}}(\Oe)} \\
&\ls \eps^{\lambda} \|\nabla(\vv_\eps - \vu_\eps)\|_{L^2(\Oe)} + \eps^{\lambda-2 + \frac{6(\gamma-1)}{2\gamma}} \|\nabla(\vv_\eps - \vu_\eps)\|^2_{L^2(\Oe)}\\
&\leq  C_\delta \eps^{2\lambda -2} + \eps^2 \delta\|\nabla (\vv_\eps - \vu_\eps)\|_{L^2(\Oe)}^2 + C \eps^{\lambda-2 + \frac{6(\gamma-1)}{2\gamma}} \|\nabla(\vv_\eps - \vu_\eps)\|^2_{L^2(\Oe)},
\end{aligned}
\end{align}
where we used 
\eqref{interpolation} with $p=2 \gamma/(\gamma-1)$, the a priori estimates \eqref{bds} as well as \eqref{v_eps}--\eqref{nabla.v_eps}, and the Poincar\'e inequality \eqref{Poinc}.

Alternatively, we may employ similar estimates to deduce 
\begin{align} 
    |I_1|&\leq 
\bigg| \eps^\lambda \int_\Oe r_\eps (\partial_t \vv_\eps + (\vu_\eps \cdot \nabla) \vv_\eps) \cdot (\vv_\eps - \vu_\eps) \dd x \bigg| \nonumber\\
&\quad+ \bigg| \eps^\lambda \int_\Oe (\rho_\eps - r_\eps) (\partial_t \vv_\eps + (\vv_\eps \cdot \nabla) \vv_\eps) \cdot (\vv_\eps - \vu_\eps) \dd x \bigg| \nonumber \\
&\quad+ \bigg| \eps^\lambda \int_\Oe (\rho_\eps - r_\eps) ((\vu_\eps - \vv_\eps) \cdot \nabla) \vv_\eps \cdot (\vv_\eps - \vu_\eps) \dd x \bigg|  \nonumber\\
&\leq  \eps^\lambda  \|\partial_t \vv_\eps + (\vu_\eps \cdot \nabla) \vv_\eps \|_{L^2(\Oe)} \|\vv_\eps -\vu_\eps\|_{L^2(\Oe)}  + \eps^{\lambda-1} \|\rho_\eps - r_\eps\|_{L^2(\Oe)}  \| \vv_\eps -\vu_\eps \|_{L^2(\Oe)} \label{I_1.2} \\
&\quad + \eps^{\lambda -1} \|\rho_\eps - r_\eps\|_{L^2(\Oe)} \| \vv_\eps -\vu_\eps \|^2_{L^4(\Oe)}  \nonumber\\
&\ls \eps^{\lambda} (1 + \|\rho_\eps - r_\eps\|_{L^2(\Oe)}) \|\nabla(\vv_\eps - \vu_\eps)\|_{L^2(\Oe)}  + \eps^{\lambda - \frac 1 2} \|\rho_\eps - r_\eps\|_{L^2(\Oe)} \| \nabla(\vv_\eps -\vu_\eps) \|^2_{L^2(\Oe)} \nonumber \\
&\leq  C_\delta \eps^{2\lambda -2} (1 + \|\rho_\eps - r_\eps\|^2_{L^2(\Oe)})  + \eps^2 \delta\|\nabla (\vv_\eps - \vu_\eps)\|_{L^2(\Oe)}^2 + C \eps^{\lambda - \frac 1 2} \|\rho_\eps - r_\eps\|_{L^2(\Oe)} \| \nabla(\vv_\eps -\vu_\eps) \|^2_{L^2(\Oe)}. \nonumber
\end{align}

We will use both estimates on $I_1$ in order to get better values on $\lambda$ dependent on the value of $\gamma$, see Section~\ref{sec:conclusion}.

Next, we observe that by the assumption $\vc f \in L^\infty((0,T) \times \Omega)$ and \eqref{v_eps}, we have, using again the Poincar\'e inequality \eqref{Poinc} and \eqref{rho_cor},
\begin{align} \label{I_2+4}
\begin{split}
    |I_2| + |I_4| &\ls \|(\rho_\eps - r_\eps) + (r_\eps - \rho)\|_{L^2(\Oe)} \|\vu_\eps - \vv_\eps\|_{L^2(\Oe)} + \|\rho_\eps - r_\eps\|_{L^2(\Oe)} \|\vu_\eps - \vv_\eps\|_{L^2(\Oe)}\\
    &\leq \delta \eps^2 \|\nabla (\vu_\eps - \vv_\eps)\|_{L^2(\Oe)}^2 +C_\delta \|r_\eps - \rho_\eps\|_{L^2(\Oe)}^2 + C_\delta \eps^2.
\end{split}
\end{align}

Regarding $I_3$, we rewrite $I_3 = I_{3,1} + I_{3,2} + I_{3,3}$ with
\begin{align*}
    I_{3,1} &= -\int_\Oe \frac 1 \theta (p(\rho_\eps)- p'(r_\eps)(\rho_\eps -r_\eps) - p(r_\eps)) \dv\vu \dd x, \\
    I_{3,2} &= \int_\Oe \frac 1 \theta p'(r_\eps) (\rho_\eps - r_\eps)\left( \frac{\rho}{r_\eps} - 1 \right) \dv \vu \dd x, \\
    I_{3,3} &= \int_\Oe \dv \bigg(\frac 1 \theta \vu - \vv_\eps \bigg) (p(\rho_\eps) - p(r_\eps)) \dd x.
\end{align*}
For $I_{3,1}$, we use $\|\dv \vu\|_\infty \lesssim 1$ and \eqref{diff.H} in combination with \eqref{r_eps} to deduce
\begin{align}\label{I_31}
    |I_{3,1}| \ls \|p(\rho_\eps)- p'(r_\eps)(\rho_\eps -r_\eps) - p(r_\eps)\|_{L^1(\Oe)} \ls \|\rho_\eps - r_\eps\|_{L^2(\Oe)}^2 + \|p(\rho_\eps) \1_{\{\rho_\eps \geq 2 r_\eps\}}\|_{L^1(\Oe)}. 
\end{align}
Similarly, using that $\|p'(r_\eps) \dv \vu \|_{L^\infty((0,T) \times \Oe)}$ is uniformly bounded, $r_\eps$ is uniformly bounded below and above, and rewriting
\begin{align*}
    (\rho_\eps - r_\eps) \Big(1-\frac{\rho}{r_\eps}\Big) = \frac{1}{r_\eps}(\rho_\eps - r_\eps) ( r_\eps - \rho),
\end{align*}
we deduce in view of \eqref{r_eps} and \eqref{rho_cor}
\begin{align}\label{I_32}
    |I_{3,2}| \lesssim \|r_\eps - \rho_\eps\|_{L^2(\Oe)}^2 + \|r_\eps - \rho\|_{L^2(\Oe)}^2 \lesssim \|r_\eps - \rho_\eps\|_{L^2(\Oe)}^2 + \eps^2.
\end{align}

For $I_{3,3}$, we first consider the term containing $p(r_\eps)$. Since $\vu$ is a regular solution to the limit system \eqref{limit.system} and $\nabla r_\eps$ is uniformly bounded by \eqref{v_eps}, we deduce together with \eqref{W^eps.H^-1} and $\dv W^\eps = 0$
\begin{align*}
    \int_\Oe p(r_\eps) \dv \bigg(\frac 1 \theta \vu - \vv_\eps \bigg) \dd x = \int_\Oe p(r_\eps) \bigg(\frac 1 \theta \Id - W^\eps \bigg):\nabla \vu \dd x \lesssim \eps \|p(r_\eps) \nabla \vu \|_{W^{1,1}(\Oe)} \lesssim \eps.
\end{align*}
To handle the remainder of $I_{3,3}$, we use $P_\eps = p(\rho_\eps)$ in $\Omega_\eps$, where $P_\eps$ is the pressure extension from Lemma \ref{lem:pressure.decomp},  yielding
\begin{align*}
    &\int_0^\tau \int_{\Omega_\eps} p(\rho_\eps) \bigg(\frac 1 \theta \Id - W^\eps \bigg):\nabla \vu \dd x \dd t = \int_0^\tau \int_{\Omega} P_\eps \bigg(\frac 1 \theta \Id - W^\eps \bigg): \nabla \vu \ \1_{\Omega_\eps} \dd x \dd t.
\end{align*}
Recalling from \eqref{i2} the notation $\| \cdot \|_{L_\eps^q} = \eps \| \cdot \|_{L^q(\Omega)} + \| \cdot \|_{[W^{1,q'}(\Omega)]'}$, we have by Lemma~\ref{lem:r.v} and since  $\nabla \vu \in W^{1,\infty}(0,T; L^\infty(\Omega))$
\begin{align}\label{i1}
    \bigg\| \bigg(\frac 1 \theta \Id - W^\eps \bigg): \nabla \vu \ \1_{\Omega_\eps} \bigg\|_{W^{1,\infty}(0,T;L_\eps^\infty(\Omega))} \lesssim \eps.
\end{align}
%
Hence, Lemma \ref{lem:pressure.decomp} with $p = \tfrac{2\gamma}{\gamma-1}$ yields 
\begin{align*}
    \left|\int_0^\tau \int_{\Omega_\eps} p(\rho_\eps) \bigg(\frac 1 \theta \Id - W^\eps \bigg):\nabla \vu \dd x \dd t\right| \lesssim \eps + \eps^{\lambda - \frac 3 \gamma} + \eps^\lambda + \eps^{\frac{\lambda}{2} +1} \lesssim \eps + \eps^{\lambda - \frac 3 \gamma},
\end{align*}
so that in total
\begin{align}\label{I_33}
    \bigg|\int_0^\tau I_{3,3} \dd t \bigg| \ls \eps + \eps^{\lambda - \frac 3 \gamma} .
\end{align}
Combining \eqref{I_31}, \eqref{I_32}, and \eqref{I_33}, we conclude
\begin{align} \label{I_3}
    \bigg| \int_0^\tau I_3 \dd t \bigg| \lesssim \eps + \eps^{\lambda - \frac3\gamma} + \int_0^\tau \|r_\eps - \rho_\eps\|_{L^2(\Oe)}^2 +\|p(\rho_\eps) \1_{\{\rho_\eps \geq 2 r_\eps\}}\|_{L^1(\Oe)} \dd t.
\end{align}

Very similarly, we estimate $I_5$. We split again $I_5 = I_{5,1} + I_{5,2} +  I_{5,3} + I_{5,4} + I_{5,5}$ with
\begin{align*}
    I_{5,1} &= \int_\Oe \left( 1 - \frac{\rho_\eps}{r_\eps} \right)p'(r_\eps)\partial_t (r_\eps -\rho) \dd x, \\
    I_{5,2} &= \int_\Oe \frac{1}{r_\eps} \bigg(W^\eps - \frac 1 \theta \Id \bigg) \vu \cdot \nabla \rho (p(\rho_\eps) - p'(r_\eps)(\rho_\eps - r_\eps) - p(r_\eps)) \dd x ,\\
    I_{5,3} &= -\int_\Oe \frac{1}{r_\eps} \bigg(W^\eps - \frac 1 \theta \Id \bigg)\vu \cdot \nabla \rho (p(\rho_\eps) - p(r_\eps)) \dd x , \\
     I_{5,4} &= \int_\Oe \frac{1}{r_\eps} W^\eps \vu \cdot \nabla (r_\eps - \rho) (p(\rho_\eps) - p'(r_\eps)(\rho_\eps - r_\eps) - p(r_\eps)) \dd x ,\\
     I_{5,5} &= -\int_\Oe \frac{1}{r_\eps} W^\eps \vu \cdot \nabla (r_\eps - \rho) (p(\rho_\eps) - p(r_\eps)) \dd x .
\end{align*}

We first estimate
\begin{align*}
   | I_{5,1}| \lesssim \|\rho_\eps - r_\eps\|_{L^2(\Oe)}^2 + \|\partial_t(r_\eps - \rho)\|_{L^2(\Oe)}^2 \lesssim \|\rho_\eps - r_\eps\|_{L^2(\Oe)}^2 + \eps^2,
\end{align*}
where we used \eqref{r_eps} and \eqref{rho_cor}. Moreover, using that $\|(W^\eps - \frac 1 \theta \Id ) \vu \cdot \nabla \rho\|_{L^\infty((0,T)\times \Oe)} \lesssim 1$ and $\|\frac{1}{r_\eps} W^\eps \vu \cdot \nabla (r_\eps - \rho)\|_{L^\infty((0,T)\times \Oe)} \lesssim 1$ due to \eqref{W.q.infty} and \eqref{r_eps}--\eqref{v_eps}, we can treat $I_{5,2}$ and $I_{5,4}$ as $I_{3,1}$ above, yielding
\begin{align*}
    | I_{5,2}| +| I_{5,4}|  \lesssim \|\rho_\eps - r_\eps\|_{L^2(\Oe)}^2 + \|p(\rho_\eps) \1_{\{\rho_\eps \geq 2 r_\eps\}}\|_{L^1(\Oe)}.
\end{align*}
Similarly, $I_{5,3}$ can be estimated exactly as $I_{3,3}$ above, furnishing
\begin{align*}
    \bigg|\int_0^\tau I_{5,3} \dd t \bigg| \ls  \eps + \eps^{\lambda - \frac 3 \gamma} .
\end{align*}

Lastly, $I_{5,5}$ is treated with the help of the pressure decomposition much as above for $I_{3,3}$. Indeed, noticing that $W^\eps \vu = 0$ on $\partial \Oe$, we may integrate by parts to the effect of
\begin{align*}
    \bigg| \int_\Oe \frac{p(r_\eps)}{r_\eps} W^\eps \vu \cdot \nabla (r_\eps - \rho) \dd x \bigg| &= \bigg|\int_\Oe (r_\eps - \rho) \dv(W^\eps \vu \ p(r_\eps)/r_\eps) \dd x \bigg|\\
    &\lesssim \|r_\eps - \rho\|_{L^2(\Omega)} \|W^\eps : \nabla (\vu \ p(r_\eps)/r_\eps)\|_{L^2(\Omega)} \lesssim \eps.
\end{align*}
The term involving $p(\rho_\eps)$ is then split through the pressure decomposition from Lemma~\ref{lem:pressure.decomp}. Indeed, using estimates in Lemma \ref{lem:r.v}, we find through integration by parts, recalling $W_\eps \vu = \vv_\eps = 0$ in $\Omega \setminus \Oe$,
\begin{align*}
    \bigg| \int_{\Omega} \frac{\vv_
    \eps}{r_\eps} \cdot \nabla (r_\eps - \rho) \varphi \dd x \bigg| = \bigg| \int_{\Oe} \dv\left(\frac{W^\eps \vu}{r_\eps} \varphi \right) \cdot (r_\eps - \rho) \dd x \bigg| \lesssim \eps \|\varphi\|_{W^{1,1}(\Omega)} \quad \text{for all } \varphi \in W^{1,1}(\Omega),
\end{align*}
that is, $\|\frac{\vv_\eps}{r_\eps} \nabla (r_\eps - \rho)\|_{[W^{1,1}(\Omega)]'} \lesssim \eps$. Hence,
\begin{align} \label{est:int.part}
    \bigg\| \frac{\vv_\eps}{r_\eps} \cdot \nabla (r_\eps - \rho) \bigg\|_{W^{1,\infty}(0,T;L_\eps^\infty(\Omega))} \lesssim \eps,
\end{align}
and by Lemma \ref{lem:pressure.decomp}, 
\begin{align*}
     \bigg|\int_0^\tau I_{5,5} \dd t \bigg| \lesssim  \eps + \eps^{\lambda - \frac 3 \gamma}.
\end{align*}

In total, we have estimated
\begin{align} \label{I_5}
    \bigg| \int_0^\tau I_5 \dd t \bigg| \lesssim \eps + \eps^{\lambda - \frac3\gamma} + \int_0^\tau \|r_\eps - \rho_\eps\|_{L^2(\Oe)}^2 + \|p(\rho_\eps) \1_{\{\rho_\eps \geq 2 r_\eps\}}\|_{L^1(\Oe)} \dd t.
\end{align}

\subsection{Conclusion} \label{sec:conclusion}

Combining estimates \eqref{I_6}, \eqref{I_2+4}, \eqref{I_3}, and \eqref{I_5} yields
\begin{align*}
    \int_0^\tau \mathcal R_\eps \dd t \leq \int_0^\tau C_\delta \|r_\eps - \rho_\eps\|_{L^2(\Oe)}^2 &+ C_\delta\|p(\rho_\eps) \1_{\{\rho_\eps \geq 2 r_\eps\}}\|_{L^1(\Oe)}  + C \delta\eps^2\|\nabla(\vu_\eps -\vv_\eps)\|_{L^2(\Oe)}^2 \dd t \\
    &+ C_\delta ( \eps^{2\lambda -2} + \eps  + \eps^{\lambda-\frac3\gamma}),
\end{align*}
for $\eps$ sufficiently small, where we also used \eqref{I_1.1} under the assumption
\begin{align}\label{lambda.1}
    \lambda-2 + \frac{6(\gamma-1)}{2\gamma} > 2 \quad \Longleftrightarrow \quad \lambda > 1 + \frac 3 \gamma.
\end{align}
In view of \eqref{r_eps}, \eqref{diff.H}, and \eqref{def.E}, we have
\begin{align} \label{E.controls.density}
      \|r_\eps - \rho_\eps\|_{L^2(\Oe)}^2 + \|p(\rho_\eps) \1_{\{\rho_\eps \geq 2 r_\eps\}}\|_{L^1(\Oe)} \lesssim E_\eps(t).
\end{align}
Thus, by \eqref{relen1} and using Korn's inequality to absorb the term $C \delta \eps^2 \|\nabla(\vu_\eps -\vv_\eps)\|_{L^2(\Oe)}^2$ by choosing $\delta$ sufficiently small, we infer for all $\tau \in [0,T]$
\begin{align*}
    E_{\eps}(\tau) + \eps^2\|\nabla(\vu_\eps -\vv_\eps)\|_{L^2((0,\tau) \times \Oe)}^2 \lesssim E_{\eps}(0) +  \int_0^\tau E_{\eps}(t) \dd t + \eps + \eps^{2\lambda -2} + \eps^{\lambda - \frac3\gamma},
\end{align*}
from which we deduce with Gr\"onwall's inequality that
\begin{align} \label{Gronwall.result}
    E_{\eps}(\tau) +  \eps^2\|\nabla(\vu_\eps -\vv_\eps)\|_{L^2((0,T) \times \Oe)}^2 \lesssim E_{\eps}(0) + \eps +  \eps^{2\lambda -2} + \eps^{\lambda - \frac3\gamma}.
\end{align}
It remains to estimate the left-hand side of \eqref{est.thm} by the left-hand side above and the right-hand side above by the right-hand side of \eqref{est.thm}. 

For the left-hand sides, we apply first \eqref{W^eps.H^-1.Omega} to $\vv_\eps -\vu = (W_\eps - \Id)\vu$ and \eqref{rho_cor} and then \eqref{diff.H} and the Poincar\'e inequality \eqref{Poinc} to deduce
\begin{align} \label{est.final.lhs}
    \begin{aligned}
    \|(\rho_\eps - \rho)(\tau)\|_{L^2(\Oe)}^2  + \|\vu_\eps -\vu\|_{L^2(0,T;W^{-1,2}(\Omega))}^2 & \lesssim \|(\rho_\eps - r_\eps)(\tau)\|_{L^2(\Oe)}^2  + \|\vu_\eps -\vv_\eps\|_{L^2((0,T) \times \Omega)}^2 + \eps^2 \\
    &\lesssim E_\eps(\tau) + \eps^2 \|\nabla (\vu_\eps - \vv_\eps)\|_{L^2((0,T) \times \Oe)}^2 + \eps^2.
    \end{aligned}
\end{align}

Regarding the right-hand side of \eqref{Gronwall.result}, we observe that
\begin{align*}
    &|H(\rho_\eps) - H'(r_\eps)(\rho_\eps - r_\eps) - H(r_\eps) - [H(\rho_\eps) - H'(\rho)(\rho_\eps - \rho) - H(\rho)]| \\
    &\leq |(H'(\rho) - H'(r_\eps))(\rho_\eps - r_\eps)| + |H(r_\eps) - H'(\rho)(r_\eps - \rho) -  H(\rho)  | \\
    &\lesssim |(\rho-r_\eps)(\rho_\eps-r_\eps)| + |\rho- r_\eps|^2 \lesssim |\rho - r_\eps|^2 + |\rho_\eps - \rho|^2 \lesssim |\rho_\eps - \rho|^2 + \eps^2,
\end{align*}
where we used that $H\in C^2$, and that both $r_\eps$ and $\rho$ are uniformly bounded above and below by \eqref{r_eps} and close to each other by \eqref{rho_cor}. Thus, using again \eqref{diff.H} as well as $\|\rho_{\eps0}\|_{L^1(\Oe)} \lesssim 1 $ by \eqref{est.initial} and $\|\vv_\eps(0,\cdot)\|_{L^\infty(\Oe)} \lesssim 1$ by \eqref{v_eps}, we have
\begin{align} \label{est.final.rhs}
\begin{aligned}
    E_\eps(0) &\lesssim \|\rho_{\eps 0} - \rho_0\|_{L^2(\Oe)}^2 + \int_\Oe H(\rho_{\eps 0}) - H'(\rho_0)(\rho_{\eps 0} - \rho_0) - H(\rho_0) \dd x + \eps^2 + \eps^\lambda \int_{\Oe} \rho_{\eps 0} |\vu_{\eps 0} - \vv_{\eps}(0,\cdot)|^2  \dd x\\
    &\lesssim  \int_\Oe H(\rho_{\eps 0}) - H'(\rho_0)(\rho_{\eps 0} - \rho_0) - H(\rho_0) \dd x + \eps^2 + \eps^\lambda \|\rho_{\eps 0} |\vu_{\eps 0}|^2\|_{L^1(\Oe)}. 
\end{aligned}
\end{align}
Inserting \eqref{est.final.lhs}--\eqref{est.final.rhs} in \eqref{Gronwall.result}  yields \eqref{est.thm}, which concludes the proof of Theorem~\ref{thm1} in the case $\Omega = \mathbb T^3$.

Alternatively, by using \eqref{I_1.2} instead of \eqref{I_1.1}, we have
\begin{align*}
    \int_0^\tau \mathcal R_\eps \dd t \leq &\int_0^\tau C_\delta (1+  \eps^{2\lambda -2}) \|r_\eps - \rho_\eps\|_{L^2(\Oe)}^2 + C_\delta\|p(\rho_\eps) \1_{\{\rho_\eps \geq 2 r_\eps\}}\|_{L^1(\Oe)}\\
    &+ C (\delta\eps^2 + \eps^{\lambda - \frac 1 2} \|\rho_\eps - r_\eps\|_{L^2(\Oe)} ) \|\nabla(\vu_\eps -\vv_\eps)\|_{L^2(\Oe)}^2 \dd t +  C_\delta ( \eps^{2\lambda -2} +  \eps  + \eps^{\lambda-\frac3\gamma}).
\end{align*}
For $\delta > 0$ define
\begin{align}\label{T.eps}
    T_\eps := \sup \left\{t \in [0,T] : \eps^{2 \lambda - 1} \sup_{s \in [0,t]} E_\eps(s) \leq  \delta^2 \eps^4 \right\}.
\end{align}
The relative entropy inequality implies
\begin{align} \label{right.lower.semicontinuity}
    \limsup_{s \downarrow t} E_\eps(s) \leq E_\eps(t) \qquad \text{for all } t \in [0,T].
\end{align}
Hence, $T_\eps > 0$ by choosing $\delta_0$ in \eqref{zusatz} sufficiently small. Then, by \eqref{E.controls.density}, there exists $\delta >0$ such that for all $\tau \in [0,T_\eps]$ we have
\begin{align*}
    E_{\eps}(\tau) + \eps^2\|\nabla(\vu_\eps -\vv_\eps)\|_{L^2((0,\tau) \times \Oe)}^2 \lesssim E_{\eps}(0) +  \int_0^\tau E_{\eps}(t) \dd t + \eps^\beta,
\end{align*}
and by Gr\"onwall's inequality
\begin{align} \label{E.T.eps}
    E_{\eps}(\tau) + \eps^2\|\nabla(\vu_\eps -\vv_\eps)\|_{L^2((0,\tau) \times \Oe)}^2 \lesssim E_{\eps}(0)  + \eps^\beta,
\end{align}
where 
 the implicit constant in this estimate is independent of $T_\eps$ since $T_\eps \leq T$, and
$\beta>0$ is as in Theorem~\ref{thm1}.
Hence, proceeding as before, we arrive at the desired estimate
\eqref{est.thm} in the time interval $(0,T_\eps)$.
The following standard continuity argument then ensures that $T_\eps = T$ for $\eps$ sufficiently small. 
Indeed, assume $T_\eps < T$. Then, \eqref{T.eps} and \eqref{right.lower.semicontinuity}  imply $\eps^{2 \lambda - 1} E_\eps(T_\eps) = \delta^2 \eps^4$. On the other hand,  \eqref{E.T.eps} combined with \eqref{zusatz} yields $\eps^{2 \lambda - 1} E_\eps(T_\eps) \lesssim \eps^4 \delta_0 <\delta^2 \eps^4$  for  $\delta_0$ sufficiently small, which is a contradiction; hence, $T_\eps = T$.

\section{Adaptations for bounded domains}\label{sec:bdDom}
\providecommand{\tve}{\tilde \vv_\eps}
In this section, we  discuss the necessary adaptations to treat smooth bounded domains $\Omega \subset \R^3$.

First, the \emph{constant} porosity $\theta$ defined in \eqref{theta} is now rather the \emph{asymptotic} porosity. Still, it obviously holds
\begin{align*}
    \lim_{\eps \to 0} \frac{|\Omega_\eps|}{|\Omega|} = \theta.
\end{align*}

Next, the main  issue is that we are no longer allowed to take $\vv_\eps$ as a test function in the relative energy since in general $\vv_\eps |_{\partial \Omega} \neq 0$. To overcome this problem, we introduce a boundary layer corrector in Proposition~\ref{pro.boundary} below. The main difficulty is that the boundary layer corrector needs to still have a bounded divergence such that the relative energy argument can be adapted. As detailed below, we obtain the boundary layer by cutting off a periodic vector potential of $W^\eps - \mathcal K$ and to use that $\vu \cdot \vc n = 0$ on $\partial \Omega$. For related but different approaches to obtain boundary layer correctors in homogenization problems for the Stokes equations, we refer to \cite{Marusic-PalokaMikelic1996, Shen2022}. \\

Recall the definition of $\vv_\eps = W^\eps \vu = W(\cdot / \eps) \mathcal{K}^{-1} \vu = W(\cdot / \eps) (\rho \vc f - \nabla p(\rho))$ from Section \ref{sec:correctors}. 

We remark that all the definitions in Section \ref{sec:correctors} remain unchanged. In particular, the estimates in Lemmas \ref{lem:r.v} and \ref{lem:diff.H} still hold except for \eqref{W^eps.H^-1}--\eqref{W^eps.H^-1.Omega}. 
Regarding \eqref{W^eps.H^-1}--\eqref{W^eps.H^-1.Omega}, we additionally need to address the cubes close to the boundary $\partial \Omega$.

To this end, note that
\begin{align*}
    \Omega_\eps \setminus \bigcup_{i=1}^{N(\eps)} Q_i^\eps \subset \{x \in \Omega : \dist(x,\partial \Omega) < C \eps\}.
\end{align*}
It follows from Lemma \ref{le:thickened.trace} below that
\begin{align*}
    \|\varphi\|_{L^1\bigl(\Omega_\eps \setminus \bigcup_{i=1}^{N(\eps)} Q_i^\eps\bigr)} \lesssim \eps \|\varphi \|_{W^{1,1}(\Omega)} \qquad \text{for all } \varphi \in W^{1,1}(\Omega).
\end{align*}
In particular, combining this with the argument from Lemma \ref{lem:r.v} for the cubes entirely contained in $\Omega$ yields
\begin{align}
   \bigg\|\left(\frac 1 \theta \Id - W^\eps\right) \1_{\Omega_\eps} \bigg\|_{[W^{1,1}(\Omega)]'} \ls \eps, \label{W^eps.Psi.boundary}\\
    \|\Id - W^\eps \|_{[W^{1,1}(\Omega)]'} \ls \eps.\label{W^eps.Psi.boundary.Omega}
\end{align}

\begin{lem}\label{le:thickened.trace} 
    Let $U \subset \R^n$ be a bounded Lipschitz domain. Then, 
    there exists a constant $C$ depending only on $U$ such that for all $\delta > 0$
    \begin{align*}
        \|\varphi\|_{L^1(U_\delta)} \leq C \delta \|\varphi\|_{W^{1,1}(U)} \quad \text{for all } \varphi \in W^{1,1}(U),
    \end{align*}
    where $U_\delta := \{ x \in U : \dist(x,\partial U) < \delta\}$.
\end{lem} 
\begin{proof}
    This is proven very similar to the trace estimate $\|\varphi \|_{L^1(\partial U)} \lesssim \|\varphi \|_{W^{1,1}(U)}$. We sketch the argument for the convenience of the reader. 

    First, it suffices to show that there exists $\delta_0 >0$ such that the estimate holds for $\delta < \delta_0$. Indeed, by choosing $C > \delta_0^{-1}$, the estimate then trivially extends to $\delta \geq \delta_0$.
    For $\delta < \delta_0$, we can  cover $U_\delta$ by finitely many balls centered on boundary points whose intersection with $U$  can be mapped by bi-Lipschitz functions to subsets of the half space. It thus suffices to prove the estimate for the half space $U = \R^n_+$,
    and by density, it is enough to consider $\varphi \in C_c^\infty(\overline {\R^n_+})$.
    Then, by the embedding $W^{1,1}(0,1) \subset C^0(0,1)$, we have for all $x \in \R^n_+$ with $x_n \leq 1$
    \begin{align*}
       |\varphi(x)| \leq \int_0^1 |\varphi(x' + t \e_n)| + |\partial_n \varphi(x' + t \e_n)| \dd t,
    \end{align*}
    where $x' = (x_1,\cdots,x_{n-1},0)$. Integrating yields for all $\delta \leq 1$
    \begin{align*}
        \int_{\{x \in \R^n_+ : x_n < \delta\}} |\varphi(x)| &\leq  \int_{\{x \in \R^n_+ : x_n < \delta\}} \int_0^1 |\varphi(x' + t \e_n)| + |\partial_n \varphi(x' + t \e_n)| \dd t \dd x \\
        &\leq \int_0^\delta \int_{\R^{n-1}} \int_0^1 |\varphi(x' + t \e_n)| + |\partial_n \varphi(x' + t \e_n)| \dd t \dd x' \dd s 
        \leq \delta \|\varphi \|_{W^{1,1}(\R_+^n)}.
    \end{align*}
\end{proof}

In order to construct a boundary corrector, we consider the problem of finding a function $\Phi : Q \to \R^{3\times3}$ such that\footnote{As in \eqref{sol.Stokes}, $Q$ shall be thought of the flat torus $\R^3 / (2 \Z)^3$.}
\begin{align*}
\left\{
\begin{aligned}
    \curl \Phi = W - \mathcal K &\quad \text{in } Q , \\
    \Phi &\quad \text{is $Q$-periodic},
\end{aligned}
\right.
\end{align*}
where the curl is taken column-wise, i.e., $(\curl \Phi)_{ij} = \mathcal E_{ikl} \partial_k \Phi_{lj}$, where $\mathcal E$ denotes the Levi-Civita symbol.

Since $\dv W = 0$ and $\int_Q W - \mathcal K \dd x = 0$, there exists a solution $\Phi \in W^{2,\infty}(Q)$ to this problem  and we may take  $\Phi = \curl F$, where $F$ is the unique $Q$-periodic solution to
\begin{align*}
    - \Delta F = W - \mathcal K \qquad \text{ in } Q.
\end{align*}
(Note that by $\dv W = 0$, it follows that $\dv F=0$ and hence $\curl \curl F = -\Delta F $.)
We observe that $F \in W^{3,\infty}(Q)$ by classical potential theory.  Indeed, we see that for any $1 \leq i,j \leq 3$ the function
$H = \partial_i \partial_j F$ can be written as $H_1 + H_2$ with
\begin{align*}
    H_1 &= (\partial_i \partial_j W)|_{Q \setminus \obst} \ast G, \\
    H_2 &= \vc n_i  (\partial_j W)|_{\partial \obst} \mathcal H^2|_{\partial \obst} \ast G,
\end{align*}
where $G$ is the fundamental solution of the Laplacian on the torus, 
$\vc n$ is the outer normal at $\partial \obst$, $\mathcal{H}^2$ is the two-dimensional Hausdorff measure, and $(\nabla W)|_{\partial \obst} $ is the trace of $\nabla W$ taken from 
 $Q \setminus \obst$. Here, we used that $\nabla W = 0$ in $\obst$.
As $(\nabla^2 W)|_{Q \setminus \obst} \in L^\infty(Q \setminus \obst)$, we have $H_1 \in W^{q,2}(Q)$ for all $q < \infty$. 
Moreover, the function $H_2$ is a single layer potential, which can be equivalently written as
\begin{align*}
    H_2(x) = \int_{\partial \obst} \vc n_i  \partial_j W(y) G(x-y) \dd y.
\end{align*}

Since $\vc n \otimes (\nabla W)|_{\partial \obst}$ is smooth, the single layer potential $H_2$ satisfies $H_2 \in W^{1,\infty}(Q)$ (see, e.g., \cite[Theorem VIII in Chapter 5]{Kellogg} for a proof in $\R^3$ instead of the torus. Since the fundamental solution on the torus and on the whole space only differ by a smooth function, the same proof applies.). This yields $F \in W^{3,\infty}(Q)$ and hence $\Phi \in W^{2,\infty}(Q)$. In particular, extending $\Phi$ periodically to the whole of $\R^3$, we have $\Phi \in W^{2,\infty}(\Omega)$.\\ 

Let $d = \dist(\obst,\partial Q) > 0$ and take a cutoff  function $\eta \in C^\infty([0,\infty))$ such that $\eta(0) = 0$ and $\eta  = 1$ on $[d,\infty)$.
Then, we define  $\eta_\eps : \Omega \to \R$ as
\begin{align*}
    \eta_\eps(x) := \eta\left(\frac{\dist(x,\partial \Omega)}\eps\right)
\end{align*}
and 
\begin{align}\label{novyWe}
    \tve(x) := \eps \curl \big( \eta_\eps \Phi(x/\eps) \big) \mathcal K^{-1} \vu + \eta_\eps \vu = \eps \nabla \eta_\eps \times \Phi(x/\eps) + \eta_\eps \vv_\eps(x).
\end{align}
We also define the boundary corrector $\Psi_\eps$ as
\begin{align}\label{bdryCorr}
   \Psi_\eps :=  \tve - \vv_\eps = \eps \curl \big( (\eta_\eps - 1) \Phi(x/\eps) \big) \mathcal K^{-1} \vu + (\eta_\eps-1) \vu.
\end{align}

Clearly, $\tve = 0$ on $\partial \Omega$, and $\tve(x) = \vv_\eps(x)$ for $\dist(x,\partial\Omega) \geq \eps d$.
In particular, $\tve = 0$ in $\Omega \setminus \Oe$.
Moreover, using that $\Phi \in W^{2,\infty}(\Omega)$ and the estimates \eqref{v_eps}--\eqref{nabla.v_eps}, we have for all $q \in [1,\infty]$ and uniformly in time
\begin{align*}
    \|\Psi_\eps\|_{L^q(\Omega)} + \eps \|\nabla \Psi_\eps\|_{L^q(\Omega)} \lesssim   \|1- \eta_\eps\|_{L^q(\Omega)} + \eps \|\nabla \eta_\eps\|_{L^q(\Omega)} + \eps^2 \|\nabla^2 \eta_\eps\|_{L^q(\Omega)}  &\lesssim \eps^\frac{1}{q}.
\end{align*}
Finally, we observe that
\begin{align*}
        \|\dv \big( \eps \curl [ \eta_\eps \Phi(x/\eps) ] \mathcal K^{-1} \vu \big) \|_{L^\infty(\Omega)} &\lesssim  1, 
\end{align*}
and
\begin{align*}
    |\dv (\eta_\eps \vu)(x)| \leq |\eta_\eps(x)| \ |\dv \vu (x)|  + |\vu(x) \cdot \nabla \eta_\eps(x)| \lesssim  1+ |\vu(x) \cdot \nabla \eta_\eps(x)|.
\end{align*}
For $\dist(x, \partial \Omega) < \eps d$, let $P x \in \partial \Omega$ be such that $\dist(x, \partial \Omega) = |x-Px|$. Note that for $\eps$ sufficiently small, $P x$ is unique and $Px - x = \dist(x, \partial\Omega) \vc n$, where $\vc n$ is the outer normal at $Px$. In particular, $\nabla \eta_\eps (x) = - \eta'(x/\eps)  \vc n/\eps$.
Hence, using $\vu(Px) \cdot \vc n = 0$ (due to \eqref{limit.system} and $\rho > 0$ in $\overline \Omega$),
\begin{align*}
    |\vu(x) \cdot \nabla \eta_\eps(x)| = \frac 1 \eps | \eta'(x/\eps) \big( \vu(x) - \vu(Px) \big) \cdot \vc n | \lesssim \frac{1}{\eps} |\nabla \vu(x)| \ |x-Px| \lesssim 1,
\end{align*}
where we used $\nabla \vu \in L^\infty(\Omega)$ in the last estimate.
In total, we deduce
\begin{align*}
    \|\dv \tve \|_{L^\infty(\Omega)} \lesssim 1.
\end{align*}
Combining this with $\|\dv \vv_\eps \|_{L^\infty(\Omega)} \lesssim 1$ and  $\tve = \vv_\eps$ for  $\dist(x, \partial \Omega) > \eps d$ yields
\begin{align*}
    \|\dv \Psi_\eps\|_{L^q(\Omega)} \lesssim \eps^\frac{1}{q}.
\end{align*}

We summarize the above properties of the boundary corrector $\Psi_\eps$ in the following proposition.

\begin{prop} \label{pro.boundary}
    For all $\eps > 0$, there exists $\Psi_\eps \in W^{1,\infty}((0,T)\times\Omega)$ such that 
    \begin{align*}
        \Psi_\eps &= - \vv_\eps \quad \text{on }  \partial \Omega, \\
        \Psi_\eps &= 0 \quad \text{in } \bigcup_{i=1}^{N(\eps)} \mathcal O_i^\eps,
    \end{align*}
    and for any $q \in [1,\infty]$,
\begin{align}
    \|\dv \Psi_{\eps}\|_{W^{1,\infty}(0,T;L^q(\Oe))} &\lesssim \eps^{\frac1q}, \label{div.Psi} \\
     \|\Psi_{\eps}\|_{W^{1,\infty}(0,T;L^q(\Oe))}  + \eps  \|\nabla \Psi_{\eps}\|_{W^{1,\infty}(0,T;L^q(\Oe))}  
     &\lesssim \eps^\frac{1}{q}, \label{Psi.q} \\
     \supp \Psi_\eps \subset \{ x \in \Omega : \dist(x,\partial \Omega) < \eps d \} &\quad \text{for $d = \dist(\obst,\partial Q) > 0$.} \label{Psi.supp}
\end{align}
 \end{prop}

We now redo the computations from Sections~\ref{sec:manip.rel.energy}--\ref{sec:conclusion} and point out the differences.\\

\paragraph{\textbf{Adaptations in Section \ref{sec:manip.rel.energy}.}}
We define $\tve$ as in \eqref{novyWe} such that $\tve = \vv_\eps + \Psi_\eps$ by \eqref{bdryCorr}. Then, by Lemma \ref{lem:r.v} and Proposition \ref{pro.boundary}, $\tilde{\vv}_\eps \in W^{1,\infty}_0((0,T) \times \Oe)$ such that $\tve$ can be used as a test function in the relative energy.\\

In the manipulation of $\mathcal{R}_\eps^2$ from \eqref{relen1}, due to linearity of $\mathbb{S}$ and the fact that $\tve - \vu_\eps|_{\partial \Oe} = 0$, we find
\begin{align*}
    \mathcal{R}_\eps^2 &= \eps^2 \int_\Oe \mathbb{S}(\nabla \vv_\eps + \nabla \Psi_\eps) : \nabla (\tve - \vu_\eps) \dd x\\
    &= \eps^2  \int_\Oe ( -\Delta W^\eps \vu -\vc z_\eps) \cdot (\tve - \vu_\eps) \dd x + \int_\Oe \mathbb{S}(\nabla \Psi_\eps) : \nabla (\tve - \vu_\eps) \dd x.
\end{align*}
The calculations afterwards are unaltered such that we arrive at
\begin{align}\label{errors.final.bd}
\begin{aligned}
     \mathcal{R}_\eps &= \eps^\lambda \int_\Oe  \rho_\eps (\partial_t \tve + (\vu_\eps \cdot \nabla) \tve) \cdot (\tve - \vu_\eps) \dd x\\
     &\quad + \int_\Oe (\rho_\eps- \rho) \vc{f} \cdot (\vu_\eps - \tve) \dd x\\
     &\quad + \int_\Oe -\left( 1 - \frac{\rho_\eps}{r_\eps} \right) \frac 1 \theta p'(r_\eps) \rho \dv \vu  - \dv \tve (p(\rho_\eps)-p(r_\eps)) \dd x \\
     &\quad + \int_\Oe \left(1 - \frac{\rho_\eps}{r_\eps} \right) \nabla p(r_\eps) \cdot (\vu_\eps - \tve) \dd x \\
     &\quad + \int_\Oe \left( 1 - \frac{\rho_\eps}{r_\eps} \right) p'(r_\eps) \bigg(\partial_t (r_\eps -\rho) + \bigg(W^\eps - \frac 1 \theta \Id \bigg) \vu \cdot \nabla \rho + W^\eps \vu \cdot \nabla(r_\eps - \rho) \bigg)\dd x \\
     &\quad + \int_\Oe \eps^2 \tilde{\vc z}_\eps \cdot (\tve - \vu_\eps) \dd x\\
      &\quad + \eps^2\int_\Oe \mathbb{S}(\nabla \Psi_\eps) : \nabla (\tve - \vu_\eps) \dd x \\
      &\quad + \int_\Oe \left( 1 - \frac{\rho_\eps}{r_\eps} \right) p'(r_\eps) \Psi_\eps \cdot \nabla r_\eps \dd x\\
     &=: \sum_{k=1}^8 I_k.
     \end{aligned}
\end{align}
Here, the additional term $I_8$ appears, because we still use \eqref{manip.R_4,5} in exactly this form (i.e., $\vv_\eps$ is not to be replaced by $\tve$) but in $I_4$ we replaced $\vv_\eps$ by $\tve$.

\medskip

\paragraph{\textbf{Adaptations in Section \ref{sec:estimate.errors}.}}
 For the additional term $I_7$, due to \eqref{Psi.q} for $q=2$ and $\delta>0$ we obtain
\begin{align*}
    \bigg| \eps^2 \int_\Oe \mathbb{S}(\nabla \Psi_\eps) : \nabla (\tve - \vu_\eps) \dd x \bigg| &\lesssim \eps^2 \|\nabla \Psi_\eps\|_{L^2(\Oe)} \|\nabla(\tve - \vu_\eps)\|_{L^2(\Oe)}\\
    &\leq\eps^2 (C_\delta \|\nabla \Psi_\eps\|_{L^2(\Oe)}^2 + \delta \|\nabla (\tve - \vu_\eps)\|_{L^2(\Oe)}^2) \leq \delta \eps^2 \|\nabla(\tve - \vu_\eps)\|_{L^2(\Oe)}^2 + C_\delta \eps.
\end{align*}

Similarly, for $I_8$, we find with $\|\nabla p(r_\eps)\|_{L^\infty(\Oe)} + \|1/r_\eps\|_{L^\infty(\Oe)} \lesssim 1$ that
\begin{align*}
    \bigg| \int_\Oe \left( 1 - \frac{\rho_\eps}{r_\eps} \right) p'(r_\eps) \Psi_\eps \cdot \nabla r_\eps \dd x \bigg| \lesssim \|\rho_\eps - r_\eps\|_{L^2(\Oe)} \|\Psi_\eps\|_{L^2(\Oe)} \lesssim \|\rho_\eps - r_\eps\|_{L^2(\Oe)}^2 + \eps.
\end{align*}

The estimates for  $I_2$, $I_4$, and $I_6$ are completely unaltered upon changing $\vv_\eps$ to $\tve$. The same holds true for $I_1$, since $\|\tve\|_{L^\infty((0,T) \times \Omega)} + \eps \|\nabla \tve\|_{L^\infty((0,T) \times \Omega)} \lesssim 1$ by the properties of $\vv_\eps$ and $\Psi_\eps$.

For $I_3$, we perform the same splitting in corresponding terms $I_{3,k}$, $k=1,2,3$, and estimate $I_{3,1}, I_{3,2}$ as before. For $I_{3,3}$, we now have
\begin{align*}
    I_{3,3} &= \int_{\Oe} \dv \bigg( \frac1\theta \vu - \tilde\vv _\eps\bigg) (p(\rho_\eps) - p(r_\eps)) \dd x\\
    &= \int_{\Oe} \bigg( \frac1\theta \Id - W^\eps \bigg):\nabla \vu (p(\rho_\eps) - p(r_\eps)) \dd x - \int_{\Oe} \dv \Psi_\eps   (p(\rho_\eps) - p(r_\eps)) \dd x.
\end{align*}
We estimate the additional term as follows
 \begin{align} \label{I_3,3.bd}
\begin{aligned}
    \left| \int_\Oe \dv \Psi_\eps (p(\rho_\eps) - p(r_\eps)) \dd x \right|
     &\leq \left| \int_\Oe \dv \Psi_\eps (p(\rho_\eps) - p'(r_\eps)(\rho_\eps - r_\eps) - p(r_\eps)) \dd x \right| 
       + \left| \int_\Oe \dv \Psi_\eps p'(r_\eps)(\rho_\eps - r_\eps)  \dd x \right| \\
    & \lesssim \|\dv \Psi_\eps\|_{L^\infty(\Omega_\eps)} \|p(\rho_\eps) - p'(r_\eps)(\rho_\eps - r_\eps) - p(r_\eps)\|_{L^1(\Omega_\eps)} \\
    &\quad +\|\dv \Psi_\eps\|_{L^2(\Omega_\eps)} \| p'(r_\eps)(\rho_\eps - r_\eps) \|_{L^2(\Omega_\eps)}  \\
    & \lesssim \eps +\|\rho_\eps - r_\eps\|_{L^2(\Oe)}^2 + \|p(\rho_\eps) \1_{\{\rho_\eps \geq 2 r_\eps\}}\|_{L^1(\Oe)} ,
\end{aligned}
 \end{align}
where we used \eqref{div.Psi} and \eqref{diff.H} in combination with \eqref{r_eps} in the last estimate. The remaining estimates for $I_3$ remain unchanged such that we have
\begin{align*}
\bigg| \int_0^\tau I_3 \dd t \bigg| \lesssim \eps + \eps^{\lambda -\frac 3 \gamma} + \int_0^\tau \|r_\eps - \rho_\eps\|_{L^2(\Oe)}^2 +\|p(\rho_\eps) \1_{\{\rho_\eps \geq 2 r_\eps\}}\|_{L^1(\Oe)} \dd t.
\end{align*}

Finally, regarding $I_5$, there is no $\tve$ involved and thus $I_{5}$ can be split into the same $\sum_{k=1}^5 I_{5,k}$ as before. The estimates for $I_{5,k}$ , $k \leq 4$, still work in the same way as before. Regarding $I_{5,5}$, though, integration by parts is no longer allowed since $W^{\eps} \vu \neq 0$ on $\partial \Omega$, and this is precisely the place where the former error $\eps$ will become $\eps^\frac{1}{2}$. Hence, we add and subtract $\Psi_\eps$ to obtain
\begin{align*}
    I_{5,5} =  \int_\Oe \frac{1}{r_\eps} \Psi_\eps \cdot \nabla (r_\eps - \rho) (p(\rho_\eps) - p(r_\eps)) \dd x -\int_\Oe \frac{1}{r_\eps} \tve \cdot \nabla (r_\eps - \rho) (p(\rho_\eps) - p(r_\eps)) \dd x =:  I_{5,5}^1  +  I_{5,5}^2 .
\end{align*}
For $I_{5,5}^1$, we treat the parts of $p(\rho_\eps)$ and $p(r_\eps)$ separately. For $p(r_\eps)$, we estimate
\begin{align*}
    \bigg| \int_\Oe \frac{1}{r_\eps} \Psi_\eps \cdot \nabla(r_\eps - \rho) p(r_\eps) \dd x \bigg| \lesssim \|\Psi_\eps\|_{L^2(\Oe)} \lesssim \eps^\frac12.
\end{align*}
For $p(\rho_\eps)$, we use the pressure decomposition from Lemma~\ref{lem:pressure.decomp}.
Since $\Psi_\eps, r_\eps^{-1}, \nabla r_\eps, \nabla \rho \in L^\infty((0,T) \times \Omega)$ uniformly in $\eps$, and $\supp \Psi_\eps \subset \{x \in \Omega : \dist(x,\partial \Omega) \leq  C \eps \}$, Lemma \ref{le:thickened.trace} yields 
\begin{align*}
   \bigg\| \frac{1}{r_\eps} \Psi_\eps \cdot \nabla (r_\eps - \rho) \bigg\|_{W^{1,\infty}(0,T;L_\eps^\infty(\Omega))}  \lesssim \eps.
\end{align*}
Hence, by Lemma~\ref{lem:pressure.decomp} for $p = \frac{2\gamma}{\gamma-1}$, we find
\begin{align*}
   \bigg| \int_0^\tau I_{5,5}^1 \dd t \bigg| =  \bigg| \int_0^\tau \int_\Oe \frac{1}{r_\eps} \Psi_\eps \cdot \nabla(r_\eps - \rho) p(\rho_\eps) \dd x \dd t \bigg|
    \lesssim  \eps^\frac12 + \eps^{\lambda - \frac{3}{\gamma}} + \eps^{\lambda} + \eps^{\frac{\lambda}{2} + 1} \lesssim \eps^\frac12 + \eps^{\lambda - \frac{3}{\gamma}}.
\end{align*}

For $I_{5,5}^2$ we proceed similarly as before when we showed \eqref{est:int.part}, which yields 
\begin{align*} 
    \bigg\| \frac{\tve}{r_\eps} \cdot \nabla (r_\eps - \rho) \bigg\|_{W^{1,\infty}(0,T;L_\eps^\infty(\Omega))} \lesssim \eps.
\end{align*}
Proceeding as for $I_{5,5}^1$, we observe that we get the same bounds for $I_{5,5}^2$. All in all, we arrive at
\begin{align*}
    \bigg| \int_0^\tau I_5 \dd t \bigg| \lesssim \eps^\frac12 + \eps^{\lambda - \frac{3}{\gamma}} + \int_0^\tau \|r_\eps - \rho_\eps\|_{L^2(\Oe)}^2 + \|p(\rho_\eps) \1_{\{\rho_\eps \geq 2 r_\eps\}}\|_{L^1(\Oe)} \dd t.
 \end{align*}

\paragraph{\textbf{Adaptations in Section \ref{sec:conclusion}.}}
Regarding Section \ref{sec:conclusion}, we combine the error estimates to obtain a similar inequality with $\vv_\eps$ replaced by $\tve$. As before, using \eqref{I_1.1} with $\vv_\eps$ replaced by $\tve$, we obtain
\begin{align*}
    \int_0^\tau \mathcal R_\eps \dd t &\leq \int_0^\tau C_\delta \|r_\eps - \rho_\eps\|_{L^2(\Oe)}^2 + C_\delta\|p(\rho_\eps) \1_{\{\rho_\eps \geq 2 r_\eps\}}\|_{L^1(\Oe)}  + C \delta\eps^2\|\nabla(\vu_\eps -\tve)\|_{L^2(\Oe)}^2 \dd t \\
    &\quad +  C_\delta ( \eps^{2\lambda -2} + \eps^\frac12  + \eps^{\lambda-\frac{3}{\gamma}}),
\end{align*}
for $\eps$ sufficiently small provided
\begin{align*}
    \lambda - 2 + \frac{6(\gamma-1)}{2\gamma} > 2.
\end{align*}
We see that this is the same condition as in \eqref{lambda.1}. In turn, all the observations made afterwards remain true. In the same spirit, using \eqref{I_1.2} with $\tve$ instead of $\vv_\eps$, we find \eqref{E.T.eps} under the additional condition \eqref{zusatz}. Since obviously, \eqref{est.final.lhs} and \eqref{est.final.rhs} still hold true (with $\vv_\eps$ replaced by $\tve$ and $\eps^2$ replaced by $\eps$), Theorem~\ref{thm1} is proven.

\section*{Acknowledgements}
R.H. warmly thanks Stefan Schiffer for discussions on extensions of divergence free vector fields.

R.H. has been supported  by the German National Academy of Science Leopoldina, grant LPDS 2020-10. \v S.N. and F.O. have been supported by the Czech Science Foundation (GA\v CR) project 22-01591S.  Moreover, \v S. N.  has been supported by  Praemium Academi{\ae} of \v S. Ne\v casov\' a. The Institute of Mathematics, CAS is supported by RVO:67985840.

 \begin{refcontext}[sorting=nyt]
\printbibliography
 \end{refcontext}
\end{document}